\documentclass[reqno, 11pt]{amsart}
\pdfoutput1

\usepackage{amssymb,color,hyperref,mathrsfs,stmaryrd}
\usepackage{amsmath}
\usepackage{mathtools}

\usepackage[usenames,dvipsnames]{xcolor}
\usepackage[curve,matrix,arrow]{xy}

\usepackage{amsmath,amsfonts,amsthm,amssymb,amscd, verbatim, graphicx}

\usepackage{pinlabel}
\usepackage{lscape}
\usepackage{tabularx}
\usepackage{latexsym}
\usepackage{hyperref}
\newtheorem{theorem}{Theorem}[section]

\newtheorem{lemma}[theorem]{Lemma}
\newtheorem{prop}[theorem]{Proposition}

\newtheorem{mprop}{Proposition}
\newtheorem{mtheorem}{Theorem}

\theoremstyle{definition}
\newtheorem{definition}[theorem]{Definition}

\numberwithin{equation}{section}

\newcommand{\Aut}{\operatorname{Aut}}

\newcommand{\Hom}{\operatorname{Hom}}

\newcommand{\id}{\operatorname{Id}}

\newcommand{\Syl}{\operatorname{Syl}}

\newcommand{\Inn}{\operatorname{Inn}}

\newcommand{\D}{\mathcal{D}}

\renewcommand{\Gamma}{\varGamma}
\renewcommand{\epsilon}{\varepsilon}

\renewcommand{\hat}{\widehat}
\renewcommand{\leq}{\leqslant}
\renewcommand{\geq}{\geqslant}

\newcommand{\<}{\langle}
\renewcommand{\>}{\rangle}

\newcommand{\oH}{\operatorname{H}}

\newcommand{\foc}{\mathfrak{foc}}
\newcommand{\hyp}{\mathfrak{hyp}}

\newcommand{\F}{\mathcal{F}}

\newcommand{\G}{\mathcal{G}}

\newcommand{\E}{\mathcal{E}}
\newcommand{\m}{\mathcal}

\newcommand{\C}{\mathcal{C}}

\newcommand{\X}{\mathcal{X}}

\newcommand{\Ac}{\operatorname{A}^\circ}

\renewcommand{\phi}{\varphi}

\begin{document}

%%
%% The title of the paper goes here.  Edit to your title.
%%

\title{Centralizers of normal subsystems revisited}
 
%%
%% Now edit the following to give your name and address:
%% 

\author{E. Henke}
\address{Institute of Mathematics, University of Aberdeen, U.K.}
\email{ellen.henke@abdn.ac.uk}
\thanks{This work was partially supported by EPSRC first grant EP/R010048/1}

\begin{abstract}
In this paper we revisit two concepts which were originally introduced by Aschbacher and are crucial in the theory of saturated fusion systems: Firstly, we give a new approach to defining the centra-lizer of a normal subsystem. Secondly, we revisit the construction of the product of two normal subsystems which centralize each other.  
\end{abstract}

\maketitle


\bigskip

\section{Introduction}

The theory of saturated fusion systems generalizes important aspects of finite group theory, since each finite group leads to a saturated fusion system which encodes the conjugacy relations between subgroups of a fixed Sylow $p$-subgroup. Much of the theory of saturated fusion systems was developed in analogy to the theory of finite groups. Building on foundational work of Puig and many other authors, Aschbacher \cite{Aschbacher:2011} introduced fusion system analogues of group theoretical concepts which play a crucial role in the proof of the classification of finite simple groups. In particular, Aschbacher \cite[Chapter~6]{Aschbacher:2011} defined centralizers of normal subsystems. In this paper we revisit this concept. We also give a new approach to the construction of the product of two normal subsystems over commuting strongly closed subgroups; such a product was first defined by Aschbacher \cite[Chapter~5]{Aschbacher:2011}. 

\smallskip

The work presented in this paper fits into a wider program to revisit and extend the theory of fusion systems. The author intends to do this partly by working with localities. Localities are group-like structures attached to fusion systems, which were introduced by Chermak \cite{Chermak:2013}, originally in the context of the proof of the existence and uniqueness of centric linking systems. Chermak and the author are in the process of developing a local theory of localities and relating fusion-theoretic concepts to analogous concepts in localities. Results about localities can then in turn be used to prove new theorems about fusion systems and to revisit existing fusion-theoretic concepts. However, it seems that some results still need to be proved in fusion systems directly, since they are necessary as a basis for relating concepts in fusion systems to their analogues in localities. The results revisited here seem to fall into this category, as they are used by Chermak and the author in \cite{Chermak/Henke:2017} to prove a one-to-one correspondence between normal subsystems of fusion systems and partial normal subgroups of certain localities. This is the motivation for this paper. We will now describe the results we prove in more detail. 

\smallskip

\textbf{For the remainder of the introduction, let $\F$ be a saturated fusion system over a $p$-group $S$.} If $R$ is any subgroup of $S$ and $\C$ is any collection of $\F$-morphisms between subgroups of $R$, write $\<\C\>_R$ for the smallest subsystem of $\F$ over $R$ containing every morphism in $\C$. 

\smallskip

Given a normal subsystem $\E$ of $\F$, Aschbacher \cite[(6.7)(1)]{Aschbacher:2011} showed that the set of subgroups $X$ of $S$ with $\E\subseteq C_\F(X)$ has a largest member $C_S(\E)$. He furthermore constructed a normal subsystem $C_\F(\E)$ over $C_S(\E)$. In Section~\ref{S:CSE}, we revisit the construction of $C_S(\E)$ by proving the theorem we state next. While part (a) of this theorem is just a reformulation of Aschbacher's result, parts (b) and (c) appear to be new. The insight gained by proving parts (b) and (c) leads to a proof of (a) which is very different from Aschbacher's proof. Our approach is actually inspired by a result concerning localities \cite[Proposition~8.2]{Henke:2015}.

\begin{mtheorem}\label{MainCSE}
Let $\E$ be a normal subsystem of $\F$ over $T\leq S$. Set
\[\X:=\{X\leq C_S(T)\colon \E\subseteq C_\F(X)\}\mbox{ and }C_S(\m{E}):=\<\X\>.
\]
Then the following hold:
\begin{itemize}
 \item[(a)] The subgroup $C_S(\E)$ is an element of $\X$, and thus with respect to inclusion the unique largest member of $\X$. Moreover, $C_S(\E)$ is strongly closed in $\F$.
 \item[(b)] The subsystem $\G:=N_{N_\F(T)}(TC_S(T))$ is a constrained subsystem of $\F$, and $N_\E(T)$ is a normal subsystem of $\G$. Thus, there exists a model $G$ for $\G$, and a normal subgroup $N$ of $G$ which is a model for $N_\E(T)$. If we fix such $G$ and $N$ and set $R^*:=C_S(N)$, then $R^*$ is -- with respect to inclusion -- the unique largest subgroup of $S$ containing $N_\E(T)$ in its centralizer in $\F$. In particular, $R^*$ does not depend on the choice of $G$ and $N$.
 \item[(c)] If $R^*$ is as in (b), then every subgroup of $R^*$ which is weakly closed in $\F$ is an element of $\X$. The subgroup $C_S(\E)$ is both the unique largest subgroup of $R^*$ which is weakly closed in $\F$, and the unique largest subgroup of $R^*$ which is strongly closed in $\F$.
\end{itemize}
\end{mtheorem}

The next Proposition appears again to be new. Recall that the focal subgroup $\foc(\F)$ of $\F$ is defined by $\foc(\F)=\<[P,\Aut_\F(P)]\colon P\leq S\>\leq S$.

\begin{mprop}\label{FocProp}
If $\E$ is a normal subsystem of $\F$ over a subgroup $T$ of $S$, then $\foc(C_\F(T))\leq C_S(\m{E})$.
\end{mprop}

Observe that Proposition~\ref{FocProp} yields $\hyp(C_\F(T))\leq C_S(\E)$, where $\hyp(\F)=\<[P,O^p(\Aut_\F(P))]\colon P\leq S\>$ denotes the hyperfocal subgroup of $\F$. This fact was already shown by Semeraro \cite[Theorem~B]{Semeraro2015} using a group theoretical result of Gross \cite{Gross:1982} and Aschbacher's (relatively technical) construction of $C_S(\E)$, which we wish to avoid. Our proof of Proposition~\ref{FocProp} uses only the characterization of $C_S(\E)$ given in Theorem~\ref{MainCSE}(a) and is essentially self-contained.

\smallskip

The observation that $\hyp(C_\F(T))\leq C_S(\E)$ leads to a new construction of the normal subsystem $C_\F(\E)$, which Aschbacher defined via machinery introduced in \cite{Aschbacher:2008}, namely constricted $\F$-invariant and normal maps. We use instead that, for every normal subgroup $R$ of $S$ with $\hyp(\F)\leq R$, the subsystem
\[\F_R:=\<O^p(\Aut_\F(P))\colon P\leq R\>_R\]
is a normal subsystem of $\F$. (In its explicit form, this result can be found in \cite[Theorem~I.7.4]{Aschbacher/Kessar/Oliver:2011}; to show that $\F_R$ is saturated, the proof builds either on work of Aschbacher using a constricted $\F$-invariant map (cf. \cite[Chapter~7]{Aschbacher:2011}), or on a different approach by Broto, Castellana, Grodal, Levi and Oliver \cite[Theorem~4.3]{BCGLO2}.)

\smallskip

Applying the above stated result with $C_\F(T)$ instead of $\F$, we can conclude that, for any normal subsystem $\E$ of $\F$ over a subgroup $T$ of $S$, the subsystem
\[C_\F(\E):=\<O^p(\Aut_{C_\F(T)}(P))\colon P\leq C_S(\E)\>_{C_S(\E)}\] 
is a normal subsystem of $C_\F(T)$. Basically, Semeraro \cite[Theorem~B]{Semeraro2015} suggests already to define $C_\F(\E)$ in this way, but he does not prove any results indicating that this is a sensible definition. We show in Proposition~\ref{Coincide} that $C_\F(\E)$ as defined above coincides with the centralizer in $\F$ of $\E$ defined by Aschbacher \cite[Chapter~6]{Aschbacher:2011}. We give moreover an independent proof that $C_\F(\E)$ is a normal subsystem of $\F$ which plays the role of a centralizer of $\E$ in $\F$. Essentially, this is the content of the following theorem.

\begin{mtheorem}\label{MainCFE}
Let $\E$ be a normal subsystem of $\F$. Then the subsystem $C_\F(\E)$ is normal in $\F$. Moreover, for any saturated subsystem $\D$ of $\F$, the two subsystems $\D$ and $\E$ centralize each other if and only if $\D$ is contained in $C_\F(\E)$. 
\end{mtheorem}

Given two subsystems $\F_1$ and $\F_2$ of $\F$ over subgroups $S_1$ and $S_2$ respectively, we say here that $\F_1$ and $\F_2$ centralize each other if $\F_i\subseteq C_\F(S_{3-i})$ for each $i=1,2$. If $\F_1$ and $\F_2$ are saturated, we show in Section~\ref{S:CentralProductBasic} that $\F_1$ and $\F_2$ centralize each other if and only if $\F$ contains a subsystem $\F_1*\F_2$ which is the central product of $\F_1$ and $\F_2$. Moreover, setting $T:=S_1S_2$, the subsystem $\F_1*\F_2$ can be explicitly constructed as the subsystem 
\[\<\psi\in\Hom_\F(P_1P_2,T)\colon P_i\leq S_i\mbox{ and }\psi|_{P_i}\in\Hom_{\F_i}(P_i,S_i)\mbox{ for }i=1,2\>_{T}.\]
If $\F_1$ and $\F_2$ are normal, then the next theorem says that this subsystem is actually normal as well.

\begin{mtheorem}\label{MainCentralProduct}
Suppose $\F_1$ and $\F_2$ are normal subsystems of $\F$ which centralize each other. Then $\F_1*\F_2$ is a normal subsystem of $\F$ and a central product of $\F_1$ and $\F_2$.
\end{mtheorem}

The result above is similar but not identical to a theorem of Aschbacher \cite[Theorem~3]{Aschbacher:2011}. Namely, suppose we are given two normal subsystems $\F_1$ and $\F_2$ over subgroups $S_1$ and $S_2$ respectively such that $[S_1,S_2]=1$. Then Aschbacher shows that there is a normal subsystem $\F_1\F_2$ of $\F$ over $S_1S_2$. If $S_1\cap S_2\leq Z(\F_i)$ for $i=1,2$, then he proves also that $\F_1\F_2$ is the central product of $\F_1$ and $\F_2$. The assumption that $S_1\cap S_2\leq Z(\F_i)$ for $i=1,2$ turns actually out to be equivalent to our assumption that $\F_1$ and $\F_2$ centralize each other; see Proposition~\ref{NormalCentralizeEachOther}. Aschbacher's theorem gives anyway an a priori stronger result, since he constructs the product $\F_1\F_2$ also under the weaker assumption that $[S_1,S_2]=1$. However, we are only interested in proving the theorem above, since this is what is needed in the work of Chermak and the author in \cite{Chermak/Henke:2017} to show a one-to-one correspondence between normal subsystems of fusion systems and partial normal subgroups of localities. With the latter result in place, it follows from a theorem on localities \cite[Theorem~1]{Henke:2015a} that a product $\F_1\F_2$ is defined in a reasonable way for any two normal subsystems $\F_1$ and $\F_2$; see \cite[Corollary~1]{Chermak/Henke:2017}. In the case that $\F_1$ and $\F_2$ centralize each other, Theorem~\ref{MainCentralProduct} gives a nice explicit description of such a product. Such an explicit description appears to be new, as Aschbacher constructs his subsystem $\F_1\F_2$ using a constricted $\F$-invariant map.

\section{Background}

\textbf{Throughout this section let $\F$ be a saturated fusion system over a finite $p$-group $S$, and let $\E$ be a subsystem of $\F$ over $T\leq S$.}

\smallskip

In this section we summarize the most significant results about fusion systems we will need. For general background on fusion systems, in particular for the definition of a saturated fusion system, we refer the reader to \cite[Chapter~I]{Aschbacher/Kessar/Oliver:2011}. We will actually build on the definition of saturation due to Broto, Levi, Oliver \cite[Definition~1.2]{Broto/Levi/Oliver:2003a}, which is stated as \cite[Proposition~I.2.5]{Aschbacher/Kessar/Oliver:2011}. We will use these properties of saturation by referring to them as the ``Sylow axiom'' and the ``extension axiom''. 

\smallskip

In addition to the notations introduced in \cite[Chapter~I]{Aschbacher/Kessar/Oliver:2011}, we will write $\F^f$ for the set of fully $\F$-normalized subgroups of $S$. We will moreover conjugate from the right and write our group homomorphisms exponentially on the right hand side. Given groups $P,Q,R$ and homomorphisms $\phi\colon P\rightarrow Q$ and $\psi\colon Q\rightarrow R$, the map $P\rightarrow R,x\mapsto (x^\phi)^\psi$ will be denoted by $\phi\psi$ or $\phi\circ\psi$. 

\smallskip

If $P,Q\leq S$, $\phi\in\Hom_\F(P,Q)$ and $\alpha\in\Hom_\F(\<P,Q\>,S)$, then we write $\phi^\alpha$ for the morphism $(\alpha|_P)^{-1}\circ\phi\circ\alpha\in\Hom_\F(P^\alpha,Q^\alpha)$. If $\alpha\in\Hom_\F(T,S)$, then $\E^\alpha$ denotes the subsystem of $\F$ over $T^\alpha$ with $\Hom_{\m{E}^\alpha}(P^\alpha,Q^\alpha)=\{\phi^\alpha\colon \phi\in\Hom_{\m{E}}(P,Q)\}$ for all $P,Q\leq T$. 

\smallskip

Throughout this text, we will often use the following facts without reference: 
\begin{itemize}
\item If $X\leq S$ and $Y\in X^\F\cap \F^f$, then there exists $\alpha\in \Hom_\F(N_S(X),S)$ with $X^{\alpha}=Y$. In particular, for every $X\leq S$, there exists $\alpha\in\Hom_\F(N_S(X),S)$ such that $X^\alpha\in\F^f$. For a proof see \cite[Lemma~I.2.6(c)]{Aschbacher/Kessar/Oliver:2011}.
\item The normalizer of a fully normalized subgroup is saturated; see \cite[Theorem~I.5.5]{Aschbacher/Kessar/Oliver:2011}.
\end{itemize}

\smallskip

It will often be useful to work with models for constrained fusion systems. Recall that $\F$ is called \emph{constrained} if there is a normal centric subgroup of $\F$. Moreover, a finite group $G$ is called a \emph{model} for $\F$ if $S$ is contained in $G$ as a Sylow $p$-subgroup, $\F=\F_S(G)$, and $C_G(O_p(G))\leq O_p(G)$. For convenience, we summarize the relationships between constrained fusion systems and models in the following theorem.

\begin{theorem}\label{Model1}
\begin{itemize}
\item [(a)] $\F$ is constrained if and only if there exists a model for $\F$. In this case, a model is unique up to an isomorphism which is the identity on $S$.
\item [(b)] If $\F$ is constrained and $G$ is a model for $\F$, then a subgroup of $S$ is normal in $\F$ if and only if it is normal in $G$. If $Q\leq S$ is normal and centric in $\F$, then in addition $C_G(Q)\leq Q$.
\item [(c)] If $\F$ is constrained, $G$ is a model for $\F$ and $\E$ is a normal subsystem of $\F$, then there exists a unique normal subgroup of $G$ which is a model for $\E$.
\end{itemize}
\end{theorem}

\begin{proof}
If $G$ is a model for $\F$ then clearly every normal $p$-subgroup of $G$ is normal in $\F$, so in particular, $\F$ is constrained. Thus, (a) follows from \cite[Theorem~III.5.10]{Aschbacher/Kessar/Oliver:2011}. Let now $\F$ be constrained and $G$ a model for $\F$. If $Q$ is a normal centric subgroup of $\F$ then it follows again from \cite[Theorem~III.5.10]{Aschbacher/Kessar/Oliver:2011} that $Q\unlhd G$ and $C_G(Q)\leq Q$. In particular, $O_p(\F)\unlhd G$. So if $g\in G$ then $c_g|_{O_p(\F)}\in \Aut_\F(O_p(\F))$ and thus $P^g=P$ for every normal subgroup $P$ of $\F$. This shows that every normal subgroup of $\F$ is normal in $G$. So (b) holds.
By \cite[Lemma~1.2(a)]{MS:2012b}, every normal subgroup $N$ of $G$ has the property $C_N(O_p(N))\leq O_p(N)$. So $N$ is a model for $\E$ if and only if $T\in\Syl_p(N)$ and $\E=\F_T(N)$. By \cite[Theorem~II.7.5]{Aschbacher/Kessar/Oliver:2011}, there exists a unique normal subgroup $N$ of $G$ with $T\in\Syl_p(N)$ and $\E=\F_T(N)$. This proves (c).
\end{proof}

We call a set $\m{C}$ of subgroups of $S$ a \emph{conjugation family} for $\F$ if, for each $P,Q\leq S$ and each $\phi\in\Hom_\F(P,Q)$, there are subgroups $P=P_0,P_1,\dots,P_k=P^\phi$, subgroups $R_i\in\m{C}$ with $\<P_{i-1},P_i\>\leq R_i$ for $i=1,\dots,r$, and automorphisms $\phi_i\in \Aut_\F(R_i)$ such that $P_{i-1}^{\phi_i}=P_i$ for each $i=1,\dots,r$ and $\phi=(\phi_1|_{P_0})\circ (\phi_2|_{P_1})\circ\dots\circ (\phi_r|_{P_{r-1}})$.

\smallskip

We will use without further reference that, by Alperin's fusion theorem, the set $\F^{cr}\cap \F^f$ of centric radical fully normalized subgroups forms a conjugation family. Indeed, the Alperin--Goldschmidt fusion theorem \cite[Theorem~I.3.6]{Aschbacher/Kessar/Oliver:2011} gives a slightly stronger statement, but Alperin's fusion theorem will be sufficient for our purposes.

\smallskip

In the remainder of this section we will collect some background results concerning subsystems of $\F$. In particular, we will be interested in $\F$-invariant and normal subsystems of $\F$. See \cite[Definition~I.6.1]{Aschbacher/Kessar/Oliver:2011} for the definition of $\F$-invariant, weakly normal and normal subsystems. We will refer to the Frattini property and the extension property stated in this definition. 

\smallskip

Next we will state some equivalent conditions for a subsystem to be $\F$-invariant. The proposition we state is basically a slight refinement of (the relevant part of)  \cite[Proposition~I.6.4]{Aschbacher/Kessar/Oliver:2011}. Some authors use part (f) of this proposition to define $\F$-invariant subsystems.

\begin{prop}\label{Finvariant}
Suppose that $T$ is strongly closed in $\F$. Then the following conditions are equivalent:
\begin{itemize}
 \item [(a)] $\E$ is $\F$-invariant.
 \item [(b)] $\E^\alpha=\E$ for each $\alpha\in\Aut_\F(T)$, and $\Aut_\E(P)\unlhd \Aut_\F(P)$ for every $P\leq T$.
 \item [(c)] $\E^\alpha=\E$ for each $\alpha\in\Aut_\F(T)$, and $\Aut_\E(P)\unlhd \Aut_\F(P)$ for every $P\leq T$ with $P\in\F^f$.
 \item [(d)] $\E^\alpha=\E$ for each $\alpha\in\Aut_\F(T)$, and $\Aut_\E(R\cap T)\unlhd \Aut_\F(R\cap T)$ for every $R\in \F^{cr}$ with $R\cap T\in \F^f$.
 \item [(e)] $\E^\alpha=\E$ for each $\alpha\in\Aut_\F(T)$, and there exists a conjugation family $\m{C}$ for $\F$ such that $\Aut_\E(R\cap T)\unlhd \Aut_\F(R\cap T)$ for every $R\in\m{C}$.
 \item [(f)] \emph{(strong invariant condition)} For each pair of subgroups $P\leq Q\leq T$, each $\phi\in\Hom_\E(P,Q)$, and each $\psi\in\Hom_\F(Q,T)$, we have $\phi^\psi\in\Hom_\E(P^\psi,Q^\psi)$.  
\end{itemize}
\end{prop}

\begin{proof}
Write $\F|_{\leq T}$ for the full subcategory of $\F$ with objects the subgroups of $T$. By \cite[Proposition~I.6.4]{Aschbacher/Kessar/Oliver:2011}, conditions (a),(b) and (f) are equivalent, and $\E$ is $\F$-invariant if and only if $\F|_{\leq T}=\<\Aut_\F(T),\E\>$ and $\E^\alpha=\E$ for each $\alpha\in\Aut_\F(T)$. Moreover, essentially the same argument as in the proof of the direction $(c)\Longrightarrow (b)$ in \cite[Proposition~I.6.4]{Aschbacher/Kessar/Oliver:2011} shows that (e) implies $\F|_{\leq T}=\<\Aut_\F(T),\E\>$; in the argument, the set of essential subgroups of $\F$ together with $S$ needs to be replaced by the arbitrary conjugation family $\m{C}$. Therefore, (e) implies (a). Clearly (b) implies (c), and (c) implies (d). So it remains to show that (d) implies (e). It suffices to argue that $\m{C}:=\{R\in\F^{cr}\colon R\cap T\in\F^f\}$ is a conjugation family. For the proof we observe that, for every $R\in\F^{cr}$, there exists $\alpha\in\Hom_\F(N_S(R\cap T),S)$ such that $(R\cap T)^{\alpha}\in\F^f$. As $R\leq N_S(R\cap T)$, the subgroup $R^\alpha\in R^\F$ is well-defined. As $\F^{cr}$ is closed under taking $\F$-conjugates, we have $R^\alpha\in\F^{cr}$. Moreover, as $T$ is strongly closed, $(R\cap T)^\alpha=R^\alpha\cap T$. So $\m{C}$ contains a representative of every $\F$-conjugacy class of centric radical subgroups. Therefore, as $\F^{cr}\cap \F^f$ is a conjugation family by Alperin's fusion theorem, it follows from \cite[Proposition~2.10]{DGMP} that $\m{C}$ is a conjugation family.
\end{proof}

As fully normalized subgroups of saturated fusion systems have particularly nice properties, the following lemma will be useful. 

\begin{lemma}\label{FfEf}
Suppose that $T$ is strongly closed. If $P\leq T$ with $P\in\F^f$ then $P\in\E^f$. 
\end{lemma}

\begin{proof}
Let $Q\in P^\E$ such that $Q\in\E^f$. Then $Q\in P^\F$. So there exists $\alpha\in\Hom_\F(N_S(Q),S)$ such that $Q^\alpha=P$. As $T$ is strongly closed, we have $N_T(Q)^\alpha\leq N_T(Q^\alpha)=N_T(P)$. Hence, as $\alpha$ is injective, $|N_T(Q)|\leq |N_T(P)|$. As $Q$ is fully $\E$-normalized, this implies that $P$ is fully $\E$-normalized.
\end{proof}

We conclude this section by stating three lemmas on normal subsystems.

\begin{lemma}\label{Wellknown}
If $\E$ is a normal subsystem of $\F$, then the set $\m{E}^{cr}$ is invariant under taking $\F$-conjugates.
\end{lemma}

\begin{proof}
Let $T\leq S$ such that $\E$ is a subsystem of $\F$ over $T$. Let $P\in\m{E}^{cr}$ be arbitrary, and let $\phi\in\Hom_\F(P,S)$. By the Frattini property, $\phi$ can be written as  the composition of a morphism in $\Hom_\E(P,T)$ with an automorphism in $\Aut_\F(T)$. Note that every element of $\Aut_\F(T)$ induces an automorphism of $\m{E}$ and thus leaves $\m{E}^{cr}$ invariant. As $\m{E}^{cr}$ is also invariant under taking $\E$-conjugates, it follows that $P^\phi\in\m{E}^{cr}$ proving the assertion.
\end{proof}

\begin{lemma}\label{LocalNormalSubsystems}
Suppose $\m{E}$ normal in $\F$ and let $Q\in\F^f$ such that $Q\leq T$. Then $Q\in \E^f$ and the subsystems $N_\F(Q)$ and $N_\E(Q)$ are saturated. Moreover, $N_\E(Q)$ is a normal subsystem of $N_\F(Q)$.
\end{lemma}

\begin{proof}
 As $Q\in\F^f$, $N_\F(Q)$ is saturated. By Lemma~\ref{FfEf} or by \cite[Lemma~3.4(5)]{Aschbacher:2008}, $Q\in\E^f$ and thus $N_\E(Q)$ is saturated as well. Clearly, $N_\E(Q)$ is a subsystem of $N_\F(Q)$. Moreover, using the characterization of $\F$-invariant subsystems given in Proposition~\ref{Finvariant}(f), one observes that $N_\E(Q)$ is $N_\F(Q)$-invariant. By \cite[Lemma~6.5]{Aschbacher:2008}, $N_T(Q)\in \E^c$. Moreover, by \cite[Lemma~6.10(3)]{Aschbacher:2008} (using \cite[Notation~6.1]{Aschbacher:2008}), for every $R\in\E^c$ and every $\alpha\in\Aut_\E(R)$, $\alpha$ extends to some $\hat{\alpha}\in\Aut_\F(RC_S(R))$ such that $[C_S(R),\hat{\alpha}]\leq Z(R)$. In particular, this is true for $R=N_T(Q)$. This shows that $N_\E(Q)$ is a normal subsystem of $N_\F(Q)$.
\end{proof}

The following technical lemma is only used in the proof of Proposition~\ref{FocProp}.

\begin{lemma}\label{PropHelp}
 Let $\E$ be a normal subsystem of $\F$ over $T\leq S$. Let $X\in\F^f$ such that $X\cap T\in\F^f\cap \E^c$ and $X\leq (X\cap T)C_S(T)$. Then $N_{N_\F(X)}(XC_S(X))$ is a constrained saturated fusion system, and $N_\E(X\cap T)$ is a normal subsystem of $N_{N_\F(X)}(XC_S(X))$.
\end{lemma}

\begin{proof}
Set $\F_X:=N_{N_\F(X)}(XC_S(X))$. As $X\in\F^f$, $N_\F(X)$ is saturated. Since $XC_S(X)$ is weakly closed in $N_\F(X)$ and thus fully $N_\F(X)$-normalized, it follows that $\F_X$ is saturated. Clearly $\F_X$ is constrained because $XC_S(X)$ is a normal centric subgroup of $\F_X$.

\smallskip

Set $Q:=X\cap T$. By assumption, $Q\in\F^f\cap \E^c$. So by \cite[(6.10)(2)]{Aschbacher:2008}, $\E(Q):=N_\E(Q)$ is a normal subsystem of the saturated and constrained fusion system $\D(Q):=N_{N_\F(Q)}(QC_S(Q))$. In particular $\E(Q)$ is saturated. By Theorem~\ref{Model1}(a),(c), there exists a model $G_Q$ for $\D(Q)$ and a normal subgroup $N_Q$ of $G_Q$ with $N_S(Q)\cap N_Q=N_T(Q)$ and $\F_{N_T(Q)}(N_Q)=\E(Q)$. Notice that $C_S(X)\leq C_S(Q)$ and by assumption, $X\leq QC_S(T)\leq QC_S(Q)$. Hence $XC_S(X)\leq QC_S(Q)$. By Theorem~\ref{Model1}(b), $QC_S(Q)$ is normal in $G_Q$. Notice also that $C_T(Q)\leq Q$ as $Q\in\E^c$. So it follows
\begin{eqnarray*}
[XC_S(X),N_Q]&\leq& [QC_S(Q),N_Q]\leq (QC_S(Q))\cap N_Q\\
&=&Q(C_S(Q)\cap N_Q)=QC_T(Q)=Q\leq X\leq XC_S(X).
\end{eqnarray*}
In particular, $N_Q$ normalizes $X$ and $XC_S(X)$. As $\E(Q)=\F_{N_T(Q)}(N_Q)$, this implies that $\E(Q)$ is contained in $\F_X$. Notice also that $Q=X\cap T$ is normal in $\F_X$. Since $\E$ is $\F$-invariant, it follows thus from the characterization of $\F$-invariant subsystems given in Proposition~\ref{Finvariant}(f) that $\m{E}(Q)$ is $\F_X$-invariant. As $\E(Q)$ is saturated, it remains only to prove the extension property for normal subsystems for the pair $(\E(Q),\F_X)$. 

\smallskip

Notice that $\F_X$ is a fusion system over $S_X:=N_S(X)$, and $\E(Q)$ is a fusion system over $T_Q:=N_T(Q)$. We have $[C_{S_X}(T_Q),N_Q]\leq [C_S(Q),N_Q]\leq C_S(Q)\cap N_Q=C_T(Q)\leq Q\leq T_Q$. Hence, any element $\alpha\in \Aut_{\E(Q)}(T_Q)=\Aut_{N_Q}(T_Q)$ extends to an element $\hat{\alpha}\in\Aut_{\F_X}(T_QC_{S_X}(T_Q))$ with $[C_{S_X}(T_Q),\hat{\alpha}]\leq T_Q$. This completes the proof that $\E(Q)=N_\E(Q)$ is normal in $\F_X$. 
\end{proof}

\section{Central products}\label{S:CentralProductBasic}

In this section we show that two saturated subsystems of a fusion system $\F$ centralize each other (in a certain sense) if and only if $\F$ contains a central product of these two subsystems. Here a central product is roughly speaking a certain homomorphic image of a direct product of fusion systems. To make this more precise, we start by recalling some basic definitions:

\smallskip

Let $\F$ and $\F'$ be fusion systems over $S$ and $S'$ respectively. We say that a group homomorphism $\alpha\colon S\rightarrow S'$ \emph{induces a morphism} from $\F$ to $\F'$ if, for each $\phi\in\Hom_\F(P,Q)$, there exists $\psi\in\Hom_{\F'}(P^\alpha,Q^\alpha)$ such that $(\alpha|_P)\circ\psi=\phi\circ(\alpha|_Q)$. For each $\phi$, the morphism $\psi$ is then uniquely determined. So if $\alpha$ induces a morphism from $\F$ to $\F'$, then $\alpha$ induces a map 
\[\alpha_{P,Q}\colon\Hom_\F(P,Q)\rightarrow \Hom_{\F'}(P^\alpha,Q^\alpha).\] 
Together with the map $P\mapsto P^\alpha$ from the set of objects of $\F$ to the set of objects of $\F'$, this gives a functor from $\F$ to $\F'$. If $\E$ is a subsystem of $\F$ over $T\leq S$, then we denote by $\E^\alpha$ the subsystem of $\F'$ over $T^\alpha$ which is the image of $\E$ under this functor. We say that $\alpha$ \emph{induces an epimorphism} from $\F$ to $\F'$ if $\F^\alpha=\F'$, i.e. if $\alpha$ is surjective and the induced map $\alpha_{P,Q}\colon \Hom_\F(P,Q)\rightarrow \Hom_{\F'}(P^\alpha,Q^\alpha)$ is surjective for all $P,Q\leq S$. If $\alpha\colon S\rightarrow S'$ induces a morphism from $\F$ to $\F'$, then the kernel $\ker(\alpha)$ is always strongly closed in $S$ with respect to $\F$. 

\smallskip

We now turn attention to direct products. Let $\F_i$ be a fusion system over $S_i$ for $i=1,2$. For each $i=1,2$ write $\pi_i\colon S_1\times S_2\rightarrow S_i,(s_1,s_2)\mapsto s_i$ for the projection map. 
Given $P_i,Q_i\leq S_i$ and $\phi_i\in\Hom_{\F_i}(P_i,Q_i)$ for each $i=1,2$, define an injective group homomorphism $\phi_1\times \phi_2\colon P_1\times P_2\rightarrow Q_1\times Q_2$ by
\[(x_1,x_2)^{(\phi_1\times \phi_2)}=(x_1^{\phi_1},x_2^{\phi_2})\]
for all $x_1\in P_1$ and $x_2\in P_2$. 

\smallskip

The \emph{direct product} $\F_1\times \F_2$ is the fusion system over $S_1\times S_2$ which is generated by the maps of the form $\phi_1\times \phi_2$ with $P_i,Q_i\leq S_i$ and $\phi_i\in\Hom_{\F_i}(P_i,Q_i)$ for $i=1,2$. By \cite[Theorem~I.6.6]{Aschbacher/Kessar/Oliver:2011}, every morphism in $\Hom_{\F_1\times \F_2}(P,Q)$ is of the form $(\phi_1\times\phi_2)|_P$ where $\phi_i\in \Hom_{\F_i}(P^{\pi_i},Q^{\pi_i})$ for $i=1,2$, and $\F_1\times \F_2$ is saturated if $\F_1$ and $\F_2$ are saturated.

\smallskip

Note that $\F_1$ and $\F_2$ can be in a natural way identified with subsystems of $\F_1\times \F_2$. To make this more precise, define $\iota_1\colon S_1\rightarrow S_1\times S_2$ by $s\mapsto (s,1)$, and $\iota_2\colon S_2\rightarrow S_1\times S_2$ by $s\mapsto (1,s)$. Then $\iota_i$ induces a morphism from $\F_i$ to $\F_1\times \F_2$ for $i=1,2$. We call the image $\F_i^{\iota_i}$ the canonical image of $\F_i$ in $\F_1\times \F_2$ and denote it by $\hat{\F}_i$. Set moreover  $\hat{S}_i=S^{\iota_i}$ for $i=1,2$. 

\smallskip

We are now in a position to state the main definition of this subsection. For a more detailed exposition on direct and central products of fusion systems we refer the reader to \cite[Sections~2.3 and 2.4]{Henke:2017}.

\begin{definition}
Suppose that $\F_1$ and $\F_2$ are subsystems of a fusion system $\F$ over $S$. Let $\F_i$ be a subsystem over $S_i\leq S$ for $i=1,2$.
\begin{itemize}
\item We say that $\F$ is the \emph{(internal) central product} of $\F_1$ and $\F_2$, if $S_1\cap S_2\leq Z(\F_i)$ for $i=1,2$, and the map 
\[\alpha\colon S_1\times S_2\rightarrow S,\;(x_1,x_2)\mapsto x_1x_2\] 
induces an epimorphism from $\F_1\times\F_2$ to $\F$ with $\hat{\F}_i^{\alpha}=\F_i$ for $i=1,2$. 

\smallskip

\noindent (The reader might want to note that $\alpha$ is a surjective homomorphism of groups if and only if $S=S_1S_2$ and $[S_1,S_2]=1$.)

\item We say that \emph{$\F_1$ centralizes $\F_2$}, or that \emph{$\F_1$ and $\F_2$ centralize each other}, if $\F_i\subseteq C_\F(S_{3-i})$ for $i=1,2$.
\item If $\F_1$ and $\F_2$ centralize each other, then define $\F_1*\F_2$ to be the subsystem of $\F$ over $S_1S_2$ generated by all morphisms $\psi\in\Hom_\F(P_1P_2,S_1S_2)$ with $P_i\leq S_i$ and $\psi|_{P_i}\in\Hom_{\F_i}(P_i,S_i)$ for $i=1,2$. 

\smallskip

\noindent (If $\F_1$ and $\F_2$ centralize each other, notice that $[S_1,S_2]=1$ and so $S_1S_2$ is a subgroup of $S$.)
\end{itemize}
\end{definition}

\begin{lemma}\label{L:F1F2Centralize}
Let $\F$ be a fusion system, and suppose $\F_i$ is a saturated subsystems of $\F$ over $S_i$ for each $i=1,2$. If $i\in\{1,2\}$ such that $\F_i\subseteq C_\F(S_{3-i})$, then $S_1\cap S_2\leq Z(\F_i)$.
\end{lemma}

\begin{proof}
Fix $i\in\{1,2\}$ such that $\F_i\subseteq C_\F(S_{3-i})$. Let $R\in\F_i^{cr}$. As $[S_1,S_2]=1$, we have $S_1\cap S_2\leq Z(S_i)\leq R$. Every automorphism in $\Aut_{\F_i}(R)$ extends to an automorphism of $RS_{3-i}$ which acts as the identity on $S_{3-i}$. Hence, $\Aut_{\F_i}(R)$ centralizes $S_1\cap S_2\leq R\cap S_{3-i}$. As $\F_i$ is by assumption saturated, it follows now from Alperin's fusion theorem that $S_1\cap S_2\leq Z(\F_i)$.
\end{proof}

\begin{prop}\label{P:F1F2Centralize}
Suppose $\F$ is a fusion system over $S$, and $\F_i$ is a saturated subsystems of $\F$ over $S_i$ for $i=1,2$. Then the following are equivalent:
\begin{itemize}
 \item [(i)] $\F_1$ and $\F_2$ centralize each other.
 \item [(ii)] $\F$ contains a subsystem $\D$ which is the central product of $\F_1$ and $\F_2$.
\end{itemize}
If one and thus both of these conditions hold, then $\F_1*\F_2$ is a central product of $\F_1$ and $\F_2$. In particular, $\F_1*\F_2$ is saturated.
\end{prop}

\begin{proof}
Suppose first that (ii) holds. Let $\D$ be a subsystem of $\F$ over a subgroup $T\leq S$ such that $\D$ is the central product of $\F_1$ and $\F_2$. Then $\alpha\colon S_1\times S_2\rightarrow T,(x_1,x_2)\mapsto x_1x_2$ induces an epimorphism from $\F_1\times \F_2$ to $\D$ with $\hat{\F}_i^{\alpha}=\F_i$ for $i=1,2$. In particular, $T=S_1S_2$ and $[S_1,S_2]=1$. From the definition of the direct product, one can easily check that $\hat{\F}_i\subseteq C_{\F_1\times \F_2}(\hat{S}_{3-i})$ for $i=1,2$. This implies $\F_i=\hat{\F}_i^{\alpha}\subseteq C_\D(\hat{S}_{3-i}^{\alpha})=C_\D(S_{3-i})\subseteq C_\F(S_{3-i})$ for each $i=1,2$. Hence (i) holds.

\smallskip

Assume now that (i) holds. Then in particular $[S_1,S_2]=1$. Set $T:=S_1S_2$ and let $\D:=\F_1*\F_2$. To complete the proof, it will be sufficient to show that $\D$ is the central product of $\F_1$ and $\F_2$. Note first that 
\[\alpha\colon S_1\times S_2\rightarrow T,(x_1,x_2)\mapsto x_1x_2\]
is a surjective group homomorphism as $[S_1,S_2]=1$ and $T=S_1S_2$. By Lemma~\ref{L:F1F2Centralize}, we have $S_1\cap S_2\leq Z(\F_i)$ for $i=1,2$. Hence, it remains only to show that $\alpha$ induces an epimorphism from $\F_1\times \F_2$ to $\D$ with $\hat{\F}_i^\alpha=\F_i$ for $i=1,2$. 

\smallskip

For $i=1,2$ let $\pi_i\colon S_1\times S_2\rightarrow S_i$ be the canonical projection map. Let $P,Q\leq S_1\times S_2$. Set $P_i=P\pi_i$ and $Q_i=Q\pi_i$ for $i=1,2$. We need to argue that there is a map
\[\Hom_{\F_1\times \F_2}(P,Q)\rightarrow \Hom_\D(P^\alpha,Q^\alpha),\phi\mapsto \psi_\phi\]
such that $(\alpha|_P)\circ\psi_\phi=\phi\circ(\alpha|_Q)$ for all $\phi\in\Hom_{\F_1\times \F_2}(P,Q)$. Fixing $\phi\in\Hom_{\F_1\times \F_2}(P,Q)$, we explain now how to construct $\psi_\phi$: It follows from the construction of $\F_1\times \F_2$ that $\phi$ is of the form $\phi=(\phi_1\times \phi_2)|_P$ for unique maps $\phi_i\in \Hom_{\F_i}(P_i,Q_i)$. For each $i=1,2$, as $\F_i\subseteq C_\F(S_{3-i})$, the morphism $\phi_i$ extends to $\hat{\phi}_i\in \Hom_\F(P_iS_{3-i},Q_iS_{3-i})$ such that $\hat{\phi}_i$ acts as the identity on $S_{3-i}$. Then $\hat{\phi}_1\hat{\phi}_2\in \Hom_\F(P_1P_2,Q_1Q_2)$ is well-defined. Moreover, for $x_1\in P_1$ and $x_2\in P_2$, we have $(x_1x_2)^{\hat{\phi}_1\hat{\phi}_2}=((x_1^{\phi_1})x_2)^{\hat{\phi}_2}=(x^{\phi_1})(x^{\phi_2})$. In particular,  $(\hat{\phi}_1\hat{\phi}_2)|_{P_i}=\phi_i\in\Hom_{\F_i}(P_i,Q_i)$ for $i=1,2$, which implies that $\hat{\phi}_1\hat{\phi}_2$ is a morphism in $\D$. Moreover, if $(x_1,x_2)\in P$, then 
$(x_1,x_2)^{\alpha\circ(\hat{\phi}_1\hat{\phi}_2)}=(x_1x_2)^{\hat{\phi}_1\hat{\phi}_2}=(x_1^{\phi_1})(x_2^{\phi_2})=(x_1,x_2)^{\phi\circ\alpha}\in Q^\alpha$. Hence, $\psi_\phi:=(\hat{\phi}_1\hat{\phi}_2)|_{P^\alpha}\in\Hom_\D(P^\alpha,Q^\alpha)$ with $(\alpha|_P)\circ\psi_\phi=\phi\circ(\alpha|_Q)$. So $\alpha$ induces a morphism from $\F_1\times \F_2$ to $\D$.

\smallskip

Note that $\alpha$ takes $\hat{S}_1=S_1\times\{1\}$ to $S_1$, and $\hat{S}_2=\{1\}\times S_2$ to $S_2$. The morphisms in $\hat{\F}_1$ are precisely the ones of the form $\phi=\phi_1\times \id_{\{1\}}$ with $\phi_1\in\Hom_{\F_1}(P_1,Q_1)$ ($P_1,Q_1\leq S$). Forming $\psi_\phi$ as above for such $\phi$, we have $\psi_\phi=\phi_1\in\Hom_{\F_1}(P_1,Q_1)$. Hence, $\hat{\F}_1^\alpha=\F_1$. Similarly, one concludes that $\hat{\F}_2^\alpha=\F_2$.

\smallskip

It remains to prove that the morphism induced by $\alpha$ is surjective. Let $P_1\leq S_1$ and $P_2\leq S_2$ be arbitrary, and let $\psi\in \Hom_\F(P_1P_2,T)$ with $\psi|_{P_i}\in\Hom_{\F_i}(P_i,S_i)$ for $i=1,2$. The subsystem $\D=\F_1*\F_2$ is by definition generated by such morphisms $\psi$, so it is sufficient to show that $\psi$ is in the image of $\alpha$.  To see this define $\phi:=(\psi|_{P_1})\times(\psi|_{P_2})\in\Hom_{\F_1\times \F_2}(P_1\times P_2,S_1\times S_2)$. Then for all $x_1\in P_1$ and $x_2\in P_2$, we have $(x_1x_2)^\psi=(x_1^{\psi|_{P_1}})(x_2^{\psi|_{P_2}})=(x_1,x_2)^{\phi\circ\alpha}=(x_1,x_2)^{\alpha\circ\psi_\phi}=(x_1x_2)^{\psi_\phi}$. Hence, $\psi=\psi_\phi$ lies in the image of $\alpha$. This completes the proof that $\D$ is a central product of $\F_1$ and $\F_2$. In particular, $\D$ is the image of $\F_1\times\F_2$ under some morphism. By \cite[Theorem~I.6.6]{Aschbacher/Kessar/Oliver:2011}, $\F_1\times \F_2$ is saturated since $\F_1$ and $\F_2$ are saturated. Moreover, by \cite[Lemma~II.5.4]{Aschbacher/Kessar/Oliver:2011}, the image of a saturated fusion system under a morphism is saturated. Hence, $\D$ is saturated.
\end{proof}

\section{The centralizer of $\E$ in $S$}\label{S:CSE}

\noindent \textbf{Let $\F$ be a saturated fusion system over $S$, and let $\m{E}$ be a normal subsystem of $\F$ over $T\leq S$. Set
\[\m{X}:=\{X\leq C_S(T)\colon \m{E}\subseteq C_\F(X)\}\mbox{ and }C_S(\E):=\<\X\>.\]}

In this section we prove Theorem~\ref{MainCSE} via a series of Lemmas.

\begin{lemma}\label{EasyCentralizer}
 Let $X\leq C_S(T)$ and $\phi\in\Hom_\F(XT,S)$. 
\begin{itemize}
\item [(a)] We have $X^\phi\leq C_S(T)$. Moreover, if $P\leq T$ and $\beta\in\Hom_\E(P,T)$, then $\beta$ is a morphism in $C_\F(X)$ if and only if $\beta^\phi\in\Hom_\E(P^\phi,T)$ is a morphism in $C_\F(X^\phi)$.
\item [(b)] If $X\in\X$ then $X^\phi\in\X$.
\item [(c)] If $N_\m{E}(T)\subseteq C_\F(X)$ then $N_{\m{E}}(T)\subseteq C_\F(X^\phi)$.
\end{itemize}
\end{lemma}

\begin{proof}
Since $X\leq C_S(T)$ and $\phi$ acts on $T$, we have $X^\phi\leq C_S(T)$. Let now $P\leq T$ and $\beta\in\Hom_\E(P,T)$. If $\beta$ is a morphism in $C_\F(X)$, then $\beta$ extends to $\hat{\beta}\in\Hom_\F(PX,S)$ with $\hat{\beta}|_X=\id_X$. It follows in this situation that $\hat{\beta}^\phi\in\Hom_\F(P^\phi X^\phi,TX^\phi)$ extends $\beta^\phi\in\Hom_\E(P^\phi,T)$ and induces the identity on $X^\phi$. Hence, if $\beta$ is a morphism in $C_\F(X)$, then $\beta^\phi$ is a morphism in $C_\F(X^\phi)$. Applying this property with $X^\phi$, $\beta^\phi$ and $\phi^{-1}$ in place of $X$, $\beta$ and $\phi$, we get that, if $\beta^\phi$ is a morphism in $C_\F(X^\phi)$, then $\beta$ is a morphism in $C_\F(X)$. This shows (a).

\smallskip

As $\m{E}$ is normal in $\F$, we have $\m{E}^\phi=\m{E}$. So for every $P\leq T$ and every $\alpha\in\Hom_\m{E}(P,T)$, $\alpha$ is of the form $\alpha=\beta^\phi$ with $\beta\in \Aut_{\m{E}}(P^{\phi^{-1}},T)$. Hence, by (a), $\alpha$ is a morphism in $C_\F(X^\phi)$ if $\m{E}\subseteq C_\F(X)$. Since $P$ and $\alpha$ were arbitrary, this yields $\m{E}\subseteq C_\F(X^\phi)$ if $\m{E}\subseteq C_\F(X)$. Hence (b) holds. 

\smallskip

As $\m{E}$ is normal in $\F$, we have $\Aut_{\m{E}}(T)^\phi=\Aut_{\m{E}}(T)$. So $\alpha\in\Aut_\m{E}(T)$ is of the form $\alpha=\beta^\phi$ with $\beta\in \Aut_{\m{E}}(T)$ and (c) follows from (a).
\end{proof}

It will be convenient to use the following notation for every subgroup $P\leq S$:
\[\Ac(P):=\{\phi\in\Aut_\F(P)\colon [P,\phi]\leq P\cap T,\;\phi|_{P\cap T}\in\Aut_\m{E}(P\cap T)\},\]
and 
\[\oH(P):=\{\phi\in\Aut_\F(P)\colon \phi\mbox{ extends to an element of }\Aut_\F(PN_T(P))\}.\]

\begin{lemma}\label{FrattiniCons}
If $P\leq S$ is fully $\F$-normalized, then $\Aut_\F(P)=\oH(P)\Ac(P)$.
\end{lemma}

\begin{proof}
Let $P\leq S$ be fully $\F$-normalized. Note that $\Ac(P)$ is a normal subgroup of $\Aut_\F(P)$, as $\m{E}$ is a normal subsystem of $\F$. Since $P$ is fully normalized, $\Aut_S(P)$ is a Sylow $p$-subgroup of $\Aut_\F(P)$ by the Sylow axiom. The Frattini argument for groups gives thus 
\[\Aut_\F(P)=N_{\Aut_\F(P)}(\Ac(P)\cap \Aut_S(P))\Ac(P).\] 
So it is sufficient to show that $N_{\Aut_\F(P)}(\Ac(P)\cap \Aut_S(P))\leq \oH(P)$. Note that $[P,N_T(P)]\leq P\cap T$ and thus $[P,\Aut_T(P)]\leq P\cap T$. If $\phi\in\Aut_T(P)$ then $\phi|_{P\cap T}\in\Aut_T(P\cap T)\leq \Aut_{\m{E}}(P\cap T)$.  Hence, $\Aut_T(P)\leq \Ac(P)\cap \Aut_S(P)$. So for any $\phi\in N_{\Aut_\F(P)}(\Ac(P)\cap \Aut_S(P))$, we have $N_T(P)\leq N_\phi$. By the extension axiom, $\phi$ extends thus to a member $\hat{\phi}\in\Hom_\F(PN_T(P),S)$. As $T$ is strongly closed, we have $(PN_T(P))\hat{\phi}\leq PN_T(P)$. This shows $\phi\in \oH(P)$ and thus $N_{\Aut_\F(P)}(\Ac(P)\cap \Aut_S(P))\leq \oH(P)$ as required.
\end{proof}

\begin{lemma}\label{XInvariant}
The set $\m{X}$ is invariant under taking $\F$-conjugates, i.e. for every $X\in\m{X}$ and every $\alpha\in\Hom_\F(X,S)$, we have $X^\alpha\in\m{X}$. 
\end{lemma}

\begin{proof}
 Assume the assertion is wrong. Then there exist $X\in\X$ and $\alpha\in\Hom_\F(X,S)$ such that $X^\alpha\not\in\m{X}$. In particular, there exist $X\in\X$, $X\leq P\leq S$ and $\alpha\in\Hom_\F(P,S)$ such that $X^\alpha\not\in\X$. We choose such a triple $(X,P,\alpha)$ such that $|P\cap T|$ is maximal.

\emph{Step~1:} We show that we can choose the triple $(X,P,\alpha)$ such that $P\in\F^f\cap \F^{cr}$ and $\alpha\in\Aut_\F(P)$. By Alperin's fusion theorem, there exist subgroups $P=P_0,\dots,P_n=P^\alpha$ of $S$ and $Q_1,\dots,Q_n\in\F^f\cap\F^{cr}$ and $\alpha_i\in\Aut_\F(Q_i)$ such that $\<P_{i-1},P_i\>\leq Q_i$ and $P_{i-1}^{\alpha_i}=P_i$ for $i=1,\dots,n$ and $\alpha=\alpha_1|_{P_0}\circ \alpha_2|_{P_1}\circ\cdots\circ \alpha_n|_{P_{n-1}}$. Setting $T_0:=P\cap T$ and $T_i=T_{i-1}^{\alpha_i}$ for $i=1,\dots,n$, we have $T_i\leq P_i\cap T\leq Q_i\cap T$ as $T$ is strongly closed. So $|Q_i\cap T|\geq |T_i|=|T_0|=|P\cap T|$ for $i=1,\dots,n$. Set $X_0:=X$ and $X_i=X_{i-1}^{\alpha_i}$ for $i=1,\dots,n$. As $X_n=X^\alpha$ is by assumption not a member of $\X$, there exists $i\in\{1,\dots,n\}$ such that $X_i\not\in\X$. Choosing such $i$ minimal, we have $X_{i-1}\in\X$ and $X_{i-1}^{\alpha_i}=X_i\not\in\X$. So replacing $(X,P,\alpha)$ by $(X_{i-1},Q_i,\alpha_i)$, we may assume that $P\in\F^f\cap\F^{cr}$ and $\alpha\in\Aut_\F(P)$. We will make this assumption from now on.

\emph{Step~2:} We reduce to the case that $\alpha\in \Ac(P)$. By Lemma~\ref{FrattiniCons}, we can write $\alpha=\gamma\beta$ with $\gamma\in\oH(P)$ and $\beta\in\Ac(P)$. If $T\leq P$ then Lemma~\ref{EasyCentralizer}(b) shows that $X^\alpha\in\X$. Hence, $T\not\leq P$ and $P$ is properly contained in the $p$-subgroup $PT$. Hence, $P<N_{PT}(P)=PN_T(P)$. The definition of $\oH(P)$ together with the maximality of $|P\cap T|$ implies now that $X':=X^\gamma\in\X$. So replacing $(X,P,\alpha)$ by $(X',P,\beta)$, we may assume $\alpha\in\Ac(P)$.

\smallskip

\emph{Step~3:} We now reach the final contradiction by showing that $X^\alpha\in\X$. As argued above, we may assume that $P\in\F^f\cap\F^{cr}$ and $\alpha\in\Ac(P)$. Then $\phi:=\alpha|_{P\cap T}\in\Aut_\m{E}(P\cap T)$. Set $Y:=(P\cap T)X$. As $X\in\X$, $\phi$ extends to $\hat{\phi}\in\Aut_\F(Y)$ with $\hat{\phi}|_X=\id_X$. By \cite[(7.18)]{Aschbacher:2011}, $P\cap T\in\E^c$ and thus $X\cap T\leq Z(T)\leq P\cap T$. Thus, $Y\cap T=(P\cap T)(X\cap T)=P\cap T$. Note that $[Y,\alpha]\leq [P,\alpha]\leq P\cap T\leq Y\cap T$. In particular, $\alpha$ normalizes $Y$ and $\hat{\phi}^{-1}\alpha\in C_{\Aut_\F(Y)}(Y/Y\cap T)\cap C_{\Aut_\F(Y)}(Y\cap T)\leq O_p(\Aut_\F(Y))$, where the last inclusion uses \cite[Lemma~A.2]{Aschbacher/Kessar/Oliver:2011}. Let $\beta\in\Hom_\F(N_S(Y),S)$ such that $Y^\beta$ is fully normalized. As argued above, we have $T\not\leq P$. So $P\cap T<T$ and thus $P\cap T<N_T(P\cap T)\leq N_T(Y)$, where the last inclusion uses $X\leq C_S(T)$. So $|N_S(Y)\cap T|=|N_T(Y)|>|P\cap T|$ and the maximality of $|P\cap T|$ yields that $X^\beta\in\X$. Similarly, $(X')^{\beta^{-1}}\in\X$ for every $X'\leq Y^\beta$ with $X'\in\X$. As $Y^\beta$ is fully normalized, $Y^\beta$ is fully automized and we conclude $(\hat{\phi}^{-1}\alpha)^\beta\in O_p(\Aut_\F(Y^\beta))\leq \Aut_S(Y^\beta)$. So there exists $s\in N_S(Y^\beta)$ with $(\hat{\phi}^{-1}\alpha)^\beta=c_s|_{Y^\beta}$. Then $\alpha|_Y=\hat{\phi}\circ (\beta|_Y)\circ (c_s|_{Y^\beta})\circ \beta^{-1}|_{Y^\beta}$. As $\hat{\phi}|_X=\id_X$, it follows $X^\alpha=((X^\beta)^s)^{\beta^{-1}}$. As argued above $X^\beta\in\X$ and so $X':=(X^\beta)^s\leq Y^\beta$ is an element of $\X$ by Lemma~\ref{EasyCentralizer}(b). So again by the above, we have $X^\alpha=(X')^{\beta^{-1}}\in\X$, which contradicts the choice of $X$ and $\alpha$. 
\end{proof}

\begin{prop}\label{WeaklyClosedCentralized} 
Let $R\leq C_S(T)$ such that $R$ is weakly $\F$-closed and $\Aut_\E(T)\subseteq C_\F(R)$. Then $R\in\X$.
\end{prop}

\begin{proof}
Assume $R\not\in\X$. Then there exists $P\leq T$ and $\phi\in\Hom_\E(P,T)$ such that $\phi$ is not a morphism in $C_\F(R)$. We choose $P$ and $\phi$ such that $|P|$ is maximal. Note that the composition of morphisms in $C_\F(R)$ is a morphism in $C_\F(R)$, and similarly a restriction of a morphism in $C_\F(R)$ is in $C_\F(R)$. Hence, by Alperin's fusion theorem, we may assume $P\in\m{E}^{cr}\cap\m{E}^f$ and $\phi\in\Aut_\E(P)$. Since by assumption $\Aut_\E(T)\subseteq C_\F(R)$, we have $P\neq T$. Let $\chi\in\Hom_\F(N_S(P),S)$ such that $Q:=P^\chi\in\F^f$.

\smallskip

\emph{Step~1:} We show that $\Aut_\E(Q)$ is contained in $C_\F(R)$. 

By Lemma~\ref{LocalNormalSubsystems}, $Q\in\m{E}^f$. So by the Sylow axiom, $\Aut_T(Q)$ is a Sylow $p$-subgroup of $\Aut_\E(Q)$. The Frattini argument for groups yields therefore 
\[\Aut_\E(Q)=O^{p^\prime}(\Aut_\E(Q))N_{\Aut_\E(Q)}(\Aut_T(Q)).\] 
As $Q\in\E^f$, by the extension axiom, every element of $N_{\Aut_\E(Q)}(\Aut_T(Q))$ extends to an element of $\Aut_\m{E}(N_T(Q))$. As $P\neq T$, $Q$ is a proper subgroup of $T$ and thus of $N_T(Q)$. So by the maximality of $|P|=|Q|$, every element of $\Aut_\E(N_T(Q))$ is a morphism in $C_\F(R)$. Thus, it is sufficient to show that $O^{p^\prime}(\Aut_\E(Q))$ lies in $C_\F(R)$. 

By Lemma~\ref{Wellknown}, $Q\in \m{E}^{cr}\cap \F^f$. By the extension axiom, every element of $\Aut_\m{E}(Q)$ extends to an $\F$-automorphism of $QC_S(Q)$. Clearly, every element of $\Aut_T(Q)$ extends to an element of $\Aut_T(QC_S(Q))$. Thus, every element of  $O^{p^\prime}(\Aut_\E(Q))=\<\Aut_T(Q)^{\Aut_\E(Q)}\>$ extends to an element of $\<\Aut_T(QC_S(Q))^{\Aut_\F(QC_S(Q))}\>$. Notice that $R\leq C_S(T)\leq C_S(Q)$ and in particular $[R,\Aut_T(QC_S(Q))]=1$. Since $R$ is weakly closed, $R$ is $\Aut_\F(QC_S(Q))$-invariant. Thus $[R,\<\Aut_T(QC_S(Q))^{\Aut_\F(QC_S(Q))}\>]=1$. Hence, every element of $O^{p^\prime}(\Aut_\E(Q))$ extends to an element of $\Aut_\F(QC_S(Q))$ which centralizes $R$. Thus, $O^{p^\prime}(\Aut_\E(R))$ lies in $C_\F(R)$. This finishes Step~1.

\smallskip

\emph{Step~2:} We reach a contradiction by showing that $\phi$ is a morphism in $C_\F(R)$. As $\m{E}$ is normal in $\F$ and $\phi\in\Aut_\E(P)$, $\psi:=\phi^\chi$ is an element of $\Aut_\E(Q)$. Hence, by Step~1, $\psi$ is a morphism in $C_\F(R)$. Moreover, by the Frattini property, we can write $\chi|_{N_T(P)}=\chi_0\circ\alpha$ for some $\chi_0\in\Hom_\E(N_T(P),T)$ and $\alpha\in\Aut_\F(T)$. Then $\phi=(\psi^{\alpha^{-1}})^{\chi_0^{-1}}$. As $T\unlhd S$ is fully $\F$-normalized, $\alpha$ extends by the extension axiom to $\hat{\alpha}\in\Aut_\F(TC_S(T))$. As $R$ is weakly closed, we have $R^{\hat{\alpha}^{-1}}=R$. So by Lemma~\ref{EasyCentralizer}(a), $\psi^{\alpha^{-1}}=\psi^{\hat{\alpha}^{-1}}$ is a morphism in $C_\F(R^{\hat{\alpha}^{-1}})=C_\F(R)$. Since $P\neq T$, $P$ is a proper subgroup of $N_T(P)$. Thus, the maximality of $|P|$ yields that $\chi_0$ is a morphism in $C_\F(R)$. Thus, $\phi$ is the composition of morphisms of $C_\F(R)$ and thus in $C_\F(R)$. This completes Step~2 and the proof of the proposition.
\end{proof}

Our main goal will be to show that $C_S(\E):=\<\X\>$ is an element of $\X$. By Lemma~\ref{XInvariant}, $\<\X\>$ is weakly closed. So by Proposition~\ref{WeaklyClosedCentralized}, it is sufficient to show that $\<\X\>$ contains $\Aut_\E(T)$ in its centralizer. We will prove this by showing that there is a unique largest subgroup of $S$ which is centralized by $N_\E(T)$. For that we will work with models for constrained subsystems. Set
\[\G:=N_{N_\F(T)}(TC_S(T)).\]

\begin{lemma}\label{GN}
The subsystem $\G$ is a constrained fusion system over $S$, and $N_\E(T)$ is a normal subsystem of $\G$. Therefore, there exists a model $G$ for $\G$, and a unique normal subgroup $N$ of $G$ such that $N$ is a model for $N_\E(T)$.
\end{lemma}

\begin{proof}
It is a special case of \cite[Lemma~6.10(2)]{Aschbacher:2008} that $\G$ is constrained and $N_\E(T)$ is a normal subsystem of $\G$; this can however also be easily seen directly. The assertion follows then from Theorem~\ref{Model1}(a),(c). 
\end{proof}

For the remainder of this section we fix a model $G$ for $\G$ and a normal subgroup $N$ of $G$, which is a model for $N_\E(T)$. Notice that this is possible by Lemma~\ref{GN}. We set
\[R^*:=C_S(N).\]
Our next goal will be to show that $R^*$ is the largest subgroup of $S$ containing $N_\E(T)$ in its centralizer. Crucial is the following lemma.

\begin{lemma}\label{CFCG0}
 Let $X\leq C_S(T)$ such that $N_{\m{E}}(T)\subseteq C_\F(X)$. Then every element $\alpha\in\Aut_\m{E}(T)$ extends to an element $\hat{\alpha}\in\Aut_\F(TC_S(T))$ such that $[TC_S(T),\hat{\alpha}]\leq T$ and $\hat{\alpha}|_X=\id_X$.
\end{lemma}

\begin{proof}
Let $\Phi_X$ be the set of all pairs $(Y,\phi)$ such that $TX\leq Y\leq TC_S(T)$, $\phi\in\Aut_\F(Y)$ is a $p^\prime$-element, $\phi|_T\in\Aut_\E(T)$, $[Y,\phi]\leq T$, $\phi|_X=\id_X$ and $\phi$ does not extend to an element $\hat{\phi}\in\Aut_\F(TC_S(T))$ with $[TC_S(T),\hat{\phi}]\leq T$. 

\smallskip

\emph{Step~1:} We show that, for any pair $(Y,\phi)\in\Phi_X$, $Y$ is not fully $\F$-normalized. To prove this by contradiction, let $(Y,\phi)\in\Phi_X$ with $Y\in\F^f$. By definition of $\Phi_X$, we have $\phi|_T\in\Aut_\E(T)$. So by definition of a normal subsystem, $\phi|_T$ extends to $\psi\in \Aut_\F(TC_S(T))$ with $[TC_S(T),\psi]\leq T$. Then in particular, $[Y,\psi]\leq T\leq Y$, so $\psi$ normalizes $Y$. So $\phi\psi^{-1}\in \Aut_\F(Y)$. As $\phi$ and $\psi$ both centralize $Y/T$, the composition $\phi\psi^{-1}$ centralizes $Y/T$. Moreover, as $\psi$ extends $\phi|_T$, we have $\phi\psi^{-1}|_T=\id_T$. So $\phi\psi^{-1}\in C:=C_{\Aut_\F(Y)}(Y/T)\cap C_{\Aut_\F(Y)}(T)$.  By \cite[Lemma~A.2]{Aschbacher/Kessar/Oliver:2011}, $C$ is a $p$-group. As $T$ is strongly closed, $T$ is normalized by $\Aut_\F(Y)$ and thus $C$ is a normal $p$-subgroup of $\Aut_\F(Y)$. Since $Y$ is fully normalized, $\Aut_S(Y)$ is a Sylow $p$-subgroup of $\Aut_\F(Y)$. So $\phi\psi^{-1}\in C\leq \Aut_S(Y)$. Thus, there exists $s\in C_S(Y/T)\cap C_S(T)$ with $\phi\psi^{-1}=c_s|_Y$. Then $\phi=(c_s|_Y)\circ \psi$ extends to $\chi:=(c_s|_{TC_S(T)})\circ \psi\in\Aut_\F(TC_S(T))$. Moreover, as $[TC_S(T),\psi]\leq T$, the automorphism of $TC_S(T)/T$ induced by $\chi$ equals the one induced by $c_s$. As $\chi|_Y=\phi$ and $\phi$ is of $p^\prime$-order, there exists $\hat{\phi}\in\<\chi\>$ such that $\hat{\phi}$ is of $p^\prime$-order and $\hat{\phi}|_Y=\phi$. Since the automorphism of $TC_S(T)/T$ induced by $\chi$ has $p$-power order, it follows $[TC_S(T),\hat{\phi}]\leq T$.  This contradicts $(Y,\phi)\in\Phi_X$.

\smallskip

\emph{Step~2:} Assuming the assertion is wrong, we show that we reach a contradiction. Note that the assertion is clearly true for every element $\alpha\in\Inn(T)$. Moreover, $\Inn(T)$ is a Sylow $p$-subgroup of $\Aut_\m{E}(T)$ by the Sylow axiom. So assuming the assertion is wrong, there exists a $p^\prime$-element $\alpha\in\Aut_\m{E}(T)$ which does not extend to a $p^\prime$-element $\hat{\alpha}\in\Aut_\F(TC_S(T))$ with $[TC_S(T),\hat{\alpha}]\leq T$ and $\hat{\alpha}|_X=\id_X$. Since $N_\E(T)\subseteq C_\F(X)$, there exists $\psi\in\Aut_\F(TX)$ with $\psi|_T=\alpha$ and $\psi|_X=\id_X$. Then the order of $\psi$ equals the order of $\alpha$. So $(TX,\psi)\in\Phi_X$ and $\Phi_X\neq\emptyset$. Thus we may choose $(Y,\phi)\in\Phi_X$ such that $|Y|$ is maximal. Let $\beta\in \Hom_\F(N_S(Y),S)$ such that $Y^\beta$ is fully normalized. 

\smallskip

By Lemma~\ref{EasyCentralizer}(c), we have $X^\beta\leq C_S(T)$ and $N_\E(T)\subseteq C_\F(X^\beta)$. As $\phi\in\Aut_\F(Y)$ is a $p^\prime$-element with $[Y,\phi]\leq T$, it follows that $\phi^\beta\in\Aut_\F(Y^\beta)$ is a $p^\prime$-element and $[Y^\beta,\phi^\beta]\leq T^\beta=T$. Moreover, as $\phi|_T\in\Aut_\E(T)$ and $\E$ is normal in $\F$, $(\phi^\beta)|_T=(\phi|_T)^\beta\in\Aut_\E(T)$. As $\phi|_X=\id_X$, we have $(\phi^\beta)|_{X^\beta}=\id_{X^\beta}$. However, since $Y^\beta$ is fully normalized, Step~1 gives that $(Y^\beta,\phi^\beta)$ is not a member of $\Phi_{X^\beta}$. So $\phi^\beta$ extends to $\xi\in\Aut_\F(TC_S(T))$ with $[TC_S(T),\xi]\leq T$. Note that, by the definition of $\Phi_X$, $Y$ must be a proper subgroup of $TC_S(T)$. So $Y<\tilde{Y}:=N_{TC_S(T)}(Y)=TN_{C_S(T)}(Y)$. Observe that 
$T=T^\beta\leq \tilde{Y}^\beta\leq TC_S(T)$ and so $[\tilde{Y}^\beta,\xi]\leq [TC_S(T),\xi]\leq 
T\leq \tilde{Y}^\beta$. In particular, $\xi$ normalizes $\tilde{Y}^\beta$. We obtain that $\tilde{\phi}:=(\xi|_{\tilde{Y}^\beta})^{\beta^{-1}}\in\Aut_\F(\tilde{Y})$ with $[\tilde{Y},\tilde{\phi}]\leq T$ and $\tilde{\phi}|_Y=\phi$. As $\phi$ is a $p^\prime$-element, there is $m\in\mathbb{N}$ such that $\tilde{\phi}^m$ is a $p^\prime$ element extending $\phi$. Since $(Y,\phi)\in\Phi_X$, it follows that $(\tilde{Y},\tilde{\phi}^m)\in\Phi_X$. However, this contradicts the maximality of $|Y|$. Therefore, the assertion must be true. 
\end{proof}

\begin{lemma}\label{FirstCharacterization}
 Let $X\leq C_S(T)$. Then $N_{\m{E}}(T)\subseteq C_\F(X)$ if and only if $X\leq R^*$. In particular, $N_{\m{E}}(T)\subseteq C_\F(R^*)$ and $R^*$ is with respect to inclusion the largest subgroup of $S$ containing $N_{\m{E}}(T)$ in its centralizer in $\F$. 
\end{lemma}

\begin{proof}
Let $X\leq C_S(T)$. If $[X,N]=1$ then clearly $N_\E(T)=\F_T(N)\subseteq C_\G(X)\subseteq C_\F(X)$. So assume now that $N_\E(T)\subseteq C_\F(X)$. We need to show that $[X,N]=1$. As $N$ is a model for $N_\E(T)$, $T$ is a Sylow $p$-subgroup of $N$. As $[X,T]=1$, it is enough to show that every $p^\prime$-element of $N$ centralizes $X$.

\smallskip

Let $n\in N$ be a $p^\prime$-element. Then $\alpha=c_n|_T\in\Aut_\E(T)$ is a $p^\prime$-automorphism. By Lemma~\ref{CFCG0}, $\alpha$ extends to $\hat{\alpha}\in\Aut_\F(TC_S(T))$ with $\hat{\alpha}|_X=\id_X$ and $[TC_S(T),\hat{\alpha}]\leq T$. As $\alpha$ is a $p^\prime$-element, replacing $\hat{\alpha}$ by a suitable power of $\hat{\alpha}$, we may assume that $\hat{\alpha}$ is a $p^\prime$-element. Note that $\hat{\alpha}$ is a morphism in $\G$. So there exists $g\in G$ with $\hat{\alpha}=c_g|_{TC_S(T)}$ and again, $g$ can be chosen to be a $p^\prime$-element. Then $gn^{-1}\in C_G(T)$, $g\in C_G(X)$ and $[TC_S(T),g]\leq T$. By Theorem~\ref{Model1}(b), $TC_S(T)$ is a normal subgroup of $G$ with $C_G(TC_S(T))\leq TC_S(T)$. So $C:=C_G(TC_S(T)/T)\cap C_G(T)$ is by \cite[Lemma~A.2]{Aschbacher/Kessar/Oliver:2011} a normal $p$-subgroup of $G$. Hence, $C\leq S$. Note also that $[TC_S(T),n]\leq TC_S(T)\cap N\leq S\cap N=T$ and thus $[TC_S(T),gn^{-1}]\leq T$. This yields $gn^{-1}\in C\leq S$ and so $g\in SN$. As $g$ is a $p^\prime$-element, it follows $g\in O^p(SN)\leq N$ and therefore $ng^{-1}\in N$. By Theorem~\ref{Model1}(b), we have $C_N(T)\leq T$. Hence, $ng^{-1}\in C_N(T)\leq T\leq C_G(X)$. As $g\in C_G(X)$, we can conclude that $n\in C_G(X)$. This shows that $[X,N]=1$ as required.
\end{proof}

%\begin{Thm}\label{Main0}
%Setting $C_S(\m{E}):=\<\X\>$, the subgroup $C_S(\E)$ is the unique largest member of $\X$. Moreover, $C_S(\E)$ is the largest weakly closed and the largest strongly closed subgroup of $R^*$.
%\end{Thm}

\begin{proof}[Proof of Theorem~\ref{MainCSE}]
Part (b) of Theorem~\ref{MainCSE} follows from Lemmas~\ref{GN} and \ref{FirstCharacterization}. We will now prove parts (a) and (c). To ease notation set $R:=\<\X\>$. By Lemma~\ref{XInvariant}, $\X$ is invariant under taking $\F$-conjugates. In particular $R$ is weakly closed. Note moreover, that for every $X\in\X$, $N_\E(T)\subseteq \E\subseteq C_\F(X)$ and thus, by Lemma~\ref{FirstCharacterization}, $X\leq R^*$.  This implies $R\leq R^*$ and so $\Aut_\E(T)=\Aut_N(T)\subseteq C_\G(R)\subseteq C_\F(R)$. Hence, by Proposition~\ref{WeaklyClosedCentralized}, $R$ is a member of $\X$. So every subgroup of $R$ is an element of $\X$, and $R$ is the unique largest member of $\X$. Since $\X$ is invariant under taking $\F$-conjugates, it follows in particular that $R$ is strongly closed. Note that every strongly closed subgroup is weakly closed. Moreover, by Proposition~\ref{WeaklyClosedCentralized}, every weakly closed subgroup of $R^*$ is in $\X$ and thus contained in $R$. So it follows that $R$ is the largest strongly closed and the largest weakly closed subgroup of $R^*$. This proves Theorem~\ref{MainCSE}.
\end{proof}

\section{The proof of Proposition~\ref{FocProp}}

In this section we prove Proposition~\ref{FocProp}. Assume the assertion is false and choose a saturated fusion system $\F$ over $S$ which is a minimal counterexample to Proposition~\ref{FocProp}. Let $\E$ be a normal subsystem of $\F$ over $T$ such that $\foc(C_\F(T))\not\leq C_S(\E)$. We proceed in four steps to reach a contradiction.

\smallskip

\emph{Step~1:} We show that there exists $Q\in\E^{cr}\cap \F^f$ and $P\leq C_S(T)$ such that $\Aut_\E(Q)\not\subseteq C_\F([P,\Aut_{C_\F(T)}(P)])$. To see this note that, as $\F$ is a counterexample and $\foc(C_\F(T))=\<[P,\Aut_{C_\F(T)}(P)]\colon P\leq C_S(T)\>$, there exists $P_0\leq C_S(T)$ such that $\hat{P}_0:=[P_0,\Aut_{C_\F(T)}(P_0)]\not\leq C_S(\m{E})$. By definition of $C_S(\E)$, we have then $\m{E}\not\subseteq C_\F(\hat{P}_0)$. By Alperin's fusion theorem, $\E$ is generated by the $\E$-automorphism groups of the elements of $\E^{cr}\cap \E^f$. So there exists $Q_0\in\E^{cr}$ such that $\Aut_\E(Q_0)$ is not in the centralizer in $\F$ of $\hat{P}_0$. Choose such $Q_0$ of maximal order. If $Q_0=T$ then $Q_0\in \F^f$, which implies that $Q=Q_0$ and $P=P_0$ have the desired properties. Assume now that $Q_0<T$. Then $Q_0<N_T(Q_0)$. Let $\phi\in\Hom_\F(N_S(Q_0),S)$ such that $Q:=Q_0^\phi\in\F^f$. By the Frattini property, we can write $\phi|_{N_T(Q_0)}=\phi_0\alpha$ where $\phi_0\in\Hom_\E(N_T(Q_0),T)$ and $\alpha\in\Aut_\F(T)$. Again by Alperin's fusion theorem, $\phi_0$ is the product of restrictions of $\E$-automorphisms of elements of $\E^{cr}$ whose order is greater or equal to $|N_T(Q_0)|$. As $Q_0<N_T(Q_0)$, the maximality of $|Q_0|$ yields that $\phi_0$ is a morphism in $C_\F(\hat{P}_0)$. So $\phi_0$ extends to $\hat{\phi}_0\in\Aut_\F(N_T(Q_0)\hat{P}_0)$ with $\hat{\phi}_0|_{\hat{P}_0}=\id_{\hat{P}_0}$. Since $T$ is fully $\F$-normalized, $\alpha$ extends by the extension axiom to $\hat{\alpha}\in\Aut_\F(TC_S(T))$. As $\hat{\alpha}$ acts on $T$, we have $P:=P_0^{\hat{\alpha}}\leq C_S(T)$. Moreover, $\Aut_{C_\F(T)}(P_0)^{\hat{\alpha}}=\Aut_{C_\F(T)}(P)$ and $\hat{P}_0^{\hat{\alpha}}=[P,\Aut_{C_\F(T)}(P)]$. Hence,  setting $\hat{\phi}=\hat{\phi}_0\hat{\alpha}$, we have $\hat{P}_0^{\hat{\phi}}=\hat{P}_0^{\hat{\alpha}}=[P,\Aut_{C_\F(T)}(P)]$. We have chosen $Q_0$ such that $\Aut_\E(Q_0)\not\subseteq C_\F(\hat{P}_0)$. As $Q=Q_0^\phi=Q_0^{\hat{\phi}}$, it follows thus from Lemma~\ref{EasyCentralizer}(a) that $\Aut_{\E}(Q)\not\subseteq C_\F(\hat{P}_0^{\hat{\phi}})=C_\F([P,\Aut_{C_\F(T)}(P)])$. Recall that $\phi$ was chosen such that $Q=Q_0^\phi\in\F^f$. Moreover, by Lemma~\ref{Wellknown}, we have $Q\in\m{E}^{cr}$. So Step~1 is complete. 

\smallskip

For the remainder of the proof we will fix $Q\in\E^{cr}\cap \F^f$ and $P\leq C_S(T)$ such that $\Aut_\E(Q)$ is not contained in the centralizer in $\F$ of $\hat{P}:=[P,\Aut_{C_\F(T)}(P)]$. Note that this is possible by Step~1.

\smallskip

\emph{Step~2:} We show that $Q$ is normal in $\F$. Suppose $Q$ is not normal in $\F$, i.e. $N_\F(Q)$ is a proper subsystem of $\F$. By Lemma~\ref{LocalNormalSubsystems} and as $Q\in\F^f$, $N_\F(Q)$ is saturated and $N_\E(Q)$ is a normal subsystem of $N_\F(Q)$ over $N_T(Q)$. As $\F$ is a minimal counterexample to our assertion, we conclude using Theorem~\ref{MainCSE}(a) that $N_\E(Q)\subseteq C_{N_\F(Q)}(\foc(C_{N_\F(Q)}(N_T(Q)))$. As $Q\leq T$, we have $C_\F(T)\subseteq C_{N_\F(Q)}(N_T(Q))$ and thus $\hat{P}\leq \foc(C_\F(T))\leq \foc (C_{N_\F(Q)}(N_T(Q)))$. In particular,  $\Aut_\E(Q)$ is contained in the centralizer in $\F$ of $\hat{P}$, contradicting the choice of $Q$ and $P$. Hence, $\F=N_\F(Q)$ and $Q$ is normal in $\F$. This finishes Step~2. Set
\[X:=QP.\]

\emph{Step~3:} We show that we can choose $P$ such that $X$ is fully normalized in $\F$. For the proof of this, let $\psi\in\Hom_\F(N_S(X),S)$ such that $X^\psi\in\F^f$. By Step~2, we have $Q\unlhd S$. Since $P\leq C_S(T)$, it follows that $T\leq N_S(X)$. Hence, as $T$ is strongly closed, $\psi$ acts on $T$. This implies that $P^\psi\leq C_S(T)$, $\Aut_{C_\F(T)}(P)^\psi=\Aut_{C_\F(T)}(P^\psi)$ and $\hat{P}^\psi=[P^\psi,\Aut_{C_\F(T)}(P^\psi)]$. Since Step~2 gives $Q^\psi=Q$, it follows from Lemma~\ref{EasyCentralizer}(a) that $\Aut_\E(Q)=\Aut_\E(Q^\psi)$ is not contained in $C_\F(\hat{P}^\psi)=C_\F([P^\psi,\Aut_{C_\F(T)}(P^\psi)])$. As $X^\psi=Q(P^\psi)$, replacing $P$ by $P^\psi$, we may assume that $X\in\F^f$. We will make this assumption from now on.

\smallskip

\emph{Step~4:} We now derive the final contradiction. Set $\F_X:=N_{N_\F(X)}(XC_S(X))$. Note that as $Q\in\m{E}^c$ we have $P\cap T\leq Z(T)\leq Q$ and thus $X\cap T=Q(P\cap T)=Q$. So by Lemma~\ref{PropHelp} and since we assume $X\in\F^f$, $\F_X$ is a constrained saturated fusion system and $N_\E(Q)$ is a normal subsystem of $\F_X$. So by Theorem~\ref{Model1}(a),(c), we can pick a model $G_X$ for $\F_X$ and a normal subgroup $N_X$ of $G_X$ such that $N_X$ is a model for $N_\E(Q)$, i.e. $N_X\cap N_S(X)=T$ and $N_\E(Q)=\F_T(N_X)$. As $X$ is normal in $\F_X$, $X$ is also normal in $G_X$ by Theorem~\ref{Model1}(b). As $T$ is strongly closed in $\F$, it follows that $Q=X\cap T$ is normal in $G_X$. As $Q$ is a centric normal subgroup of $N_\E(Q)$ and $N_X$ is a model for $N_\E(Q)$, Theorem~\ref{Model1}(b) gives $C_{N_X}(Q)\leq Q$. Thus $[N_X,C_{G_X}(Q)]\leq C_{N_X}(Q)\leq Q\leq T$. As $P\leq C_S(T)$, this implies $[N_X,C_{G_X}(Q),P]=1$. Moreover, $[P,N_X]\leq X\cap N_X=X\cap T=Q$ and thus $[P,N_X,C_{G_X}(Q)]=1$. The Three Subgroup Lemma yields now $[C_{G_X}(Q),P,N_X]=1$, i.e. $[P,C_{G_X}(Q)]$ is centralized by $N_X$. 

\smallskip

Let $\gamma\in\Aut_{C_\F(T)}(P)$. As $Q\leq T$ and $X=PQ$, $\gamma$ extends to an element $\hat{\gamma}\in\Aut_\F(X)$ with $\hat{\gamma}|_Q=\id_Q$. As $X$ is fully normalized, every element of $\Aut_\F(X)$ extends by the extension axiom to an element of $\Aut_\F(XC_S(X))$ and is thus a morphism in $\F_X$. In particular, $\hat{\gamma}$ is a morphism in $\F_X$. Thus $\hat{\gamma}=c_g|_X$ for some $g\in G_X$. As $\hat{\gamma}|_Q=\id_Q$, it follows $g\in N_{G_X}(Q)$. Hence, $\gamma=\hat{\gamma}|_P=c_g|_P$ is realized by an element of $C_{G_X}(Q)$. Since $\gamma$ was arbitrary, this shows $\hat{P}=[P,\Aut_{C_\F(T)}(P)]\leq [P,C_{G_X}(Q)]$. As we have seen above, this means that $\hat{P}$ is centralized by $N_X$. Thus, $N_\E(Q)=\F_T(N_X)\subseteq C_{\F_X}(\hat{P})\subseteq C_\F(\hat{P})$ and so $\Aut_\E(Q)\subseteq C_\F(\hat{P})$. This contradicts the choice of $P$ and $Q$, which proves that our initial assumption was false and the proposition holds.

\section{The centralizer of $\E$ in $\F$}

\textbf{Throughout this section suppose $\F$ is a saturated fusion system over $S$ and $\m{E}$ is a normal subsystem of $\F$ over $T$.}

\smallskip

We will prove Theorem~\ref{MainCFE}.

\smallskip

Recall from the introduction that, for any subgroup $R$ of $S$ and any collection $\C$ of $\F$-morphisms between subgroups of $R$, we write $\<\C\>_R$ for the smallest subsystem of $\F$ over $R$ containing every morphism in $\C$. 

\smallskip

We will use the fact that, for any subgroup $R$ of $S$ with $\hyp(\F)\leq R$, the subsystem 
\[\F_R:=\<O^p(\Aut_\F(P))\colon P\leq R \>_R\]
of $\F$ is saturated. Moreover, $\F_R$ is normal in $\F$ if and only if $R$ is normal in $S$; see \cite[Theorem~7.4]{Aschbacher/Kessar/Oliver:2011}. 

\smallskip

Recall that we have defined $C_S(\m{E})$. Moreover, we proved in Theorem~\ref{MainCSE}(a) that  $C_S(\m{E})$ is strongly closed in $\F$ and thus in $C_\F(T)$. Note also that Proposition~\ref{FocProp} gives $\hyp(C_\F(T))\leq \foc(C_\F(T))\leq C_S(\m{E})$. So by the above mentioned result, the subsystem 
\[C_\F(\m{E})=\<O^p(\Aut_{C_\F(T)}(P))\colon P\leq C_S(\m{E})\>_{C_S(\m{E})}\]
is a normal subsystem of $C_\F(T)$. We will use this to show that $C_\F(\m{E})$ is indeed normal in $\F$.  

\begin{lemma}\label{ShowWeaklyNormal}
The subsystem $C_\F(\m{E})$ is weakly normal in $\F$.
\end{lemma}

\begin{proof}
As explained above, $C_\F(\m{E})$ is normal in $C_\F(T)$, and thus in particular saturated. Therefore, it remains only to show that $C_\F(\m{E})$ is $\F$-invariant. Recall that $R:=C_S(\m{E})$ is strongly closed in $\F$. Moreover, if $\alpha\in\Aut_\F(R)$, then $T\leq C_S(R)\leq N_\alpha$. So by the extension axiom and as $T$ is strongly closed, $\alpha$ extends to an element of $\Aut_\F(RT)$ which acts on $T$. Therefore, for every $P\leq R$, we have $O^p(\Aut_{C_\F(T)}(P))^\alpha=O^p(\Aut_{C_\F(T)}(P^\alpha))$. This implies $C_\F(\m{E})^\alpha=C_\F(\m{E})$. 

\smallskip

Let now $P\leq C_S(\m{E})$ with $P\in\F^f$. By Proposition~\ref{Finvariant}, we only need to prove condition (c) in that proposition. So it remains to show that
\begin{equation*}\label{star}\tag{$*$}
\Aut_{C_\F(\m{E})}(P)\unlhd \Aut_\F(P).
\end{equation*}
As $P\in\F^f$, Lemma~\ref{FfEf} gives $P\in C_\F(\m{E})^f$. Thus, by the Sylow axiom, $\Aut_R(P)\in\Syl_p(\Aut_{C_\F(\m{E})}(P))$. Since $O^p(\Aut_{C_\F(T)}(P))\leq \Aut_{C_\F(\m{E})}(P)\leq \Aut_{C_\F(T)}(P)$, it follows that 
\[\Aut_{C_\F(\m{E})}(P)=\Aut_R(P)O^p(\Aut_{C_\F(T)}(P)).\]
For every $\alpha\in\Aut_\F(P)$, we have $T\leq C_S(P)\leq N_\alpha$. So by the extension axiom and since $T$ is strongly closed, every element of $\Aut_\F(P)$ extends to an element of $\Aut_\F(PT)$ acting on $T$. This implies that $\Aut_{C_\F(T)}(P)\unlhd \Aut_\F(P)$. As $P\in\F^f$, the Sylow axiom gives $\Aut_S(P)\in\Syl_p(\Aut_\F(P))$. So $\Aut_{C_S(T)}(P)=\Aut_S(P)\cap \Aut_{C_\F(T)}(P)$ is a Sylow $p$-subgroup of $\Aut_{C_\F(T)}(P)$, and thus $\Aut_{C_\F(T)}(P)=O^p(\Aut_{C_\F(T)}(P))\Aut_{C_S(T)}(P)$. The Frattini argument yields 
\begin{eqnarray*}
\Aut_\F(P)&=&\Aut_{C_\F(T)}(P)N_{\Aut_\F(P)}(\Aut_{C_S(T)}(P))\\
&=&O^p(\Aut_{C_\F(T)}(P))N_{\Aut_\F(P)}(\Aut_{C_S(T)}(P)).
\end{eqnarray*} 
Clearly $O^p(\Aut_{C_\F(T)}(P))\leq \Aut_{C_\F(\m{E})}(P)$ normalizes $\Aut_{C_\F(\m{E})}(P)$. By the extension axiom, every element of $N_{\Aut_\F(P)}(\Aut_{C_S(T)}(P))$ extends to an element of $\Aut_\F(PC_S(T))$, which then acts on $R$ as $R\leq C_S(T)$ is strongly closed. It follows that every element of $N_{\Aut_\F(P)}(\Aut_{C_S(T)}(P))$ normalizes $\Aut_R(P)$. As observed before, $\Aut_{C_\F(T)}(P)$ is normal in $\Aut_\F(P)$, which implies that $O^p(\Aut_{C_\F(T)}(P))$ is also normal in $\Aut_\F(P)$. As $\Aut_{C_\F(\m{E})}(P)=\Aut_R(P)O^p(\Aut_{C_\F(T)}(P))$, this yields (\ref{star}) and completes thus the proof of the assertion. 
\end{proof}

\begin{theorem}\label{CFENormal}
The subsystem $C_\F(\m{E})$ is normal in $\F$. 
\end{theorem}

\begin{proof}
Set $R:=C_S(\m{E})$ and $V:=(RT)C_S(RT)$. Note that $V\unlhd S$, and thus $N_\F(V)$ is a saturated subsystem of $\F$ over $S$. As $C_S(V)\leq V$, $N_\F(V)$ is constrained. So by Theorem~\ref{Model1}(a),(b), we may fix a model $G$ for $N_\F(V)$, and $V$ will then be a normal subgroup of $G$ with $C_G(V)\leq V$. Note that $T$ and $R$ are both normal in $G$, as they are both contained in $V$ and strongly closed in $\F$.

\smallskip

\emph{Step~1:} We show that $[C_S(R),O^p(C_G(T))]\leq R$. For the proof we will use several times that, by Proposition~\ref{FocProp}, $\foc(C_\F(T))\leq R$. Note that $V$ is normal in $G$ and $[V,C_G(RT)]=[C_S(RT),C_G(RT)]\leq [C_V(T),C_G(T)]\leq \foc(C_\F(T))\leq R$. So $[V,O^p(C_G(RT))]=[V,O^p(C_G(RT)),O^p(C_G(RT))]\leq [R,O^p(C_G(RT))]=1$. As $C_G(V)\leq V$, this implies $O^p(C_G(RT))=1$. Thus, $C_G(RT)$ is a normal $p$-subgroup, which yields $C_G(RT)=C_S(RT)$. So $[C_S(R),C_G(T)]\leq C_G(RT)=C_S(RT)\leq C_S(R)$. We conclude that $C_G(T)$ acts on $C_S(R)$ and 
\begin{eqnarray*}
[C_S(R),O^p(C_G(T))]&=&[C_S(R),O^p(C_G(T)),O^p(C_G(T))]\\
&\leq& [C_S(RT),O^p(C_G(T))]\leq \foc(C_\F(T))\leq R. 
\end{eqnarray*}
This completes Step~1. 

\smallskip

\emph{Step~2:} We show now that the assertion holds. Recall first that, by Lemma~\ref{ShowWeaklyNormal}, $C_\F(\m{E})$ is weakly normal in $\F$. Let  $\alpha\in\Aut_{C_\F(\m{E})}(R)$. It remains to prove that $\alpha$ extends to $\hat{\alpha}\in\Aut_\F(RC_S(R))$ with $[C_S(R),\hat{\alpha}]\leq R$. If $\alpha=c_s|_R$ for some $s\in R$, then this condition is clearly fulfilled with $\hat{\alpha}=c_s|_{RC_S(R)}$. By definition of $C_\F(\m{E})$ at the beginning of this section, we have $\Aut_{C_\F(\m{E})}(R)=\Inn(R)O^p(\Aut_{C_\F(T)}(R))$. So we may assume that $\alpha\in\Aut_{C_\F(T)}(R)$ is a $p^\prime$-element. By definition of $C_\F(T)$, $\alpha$ extends to $\widetilde{\alpha}\in\Aut_\F(RT)$ with $\widetilde{\alpha}|_T=\id_T$. Note that $\widetilde{\alpha}$ extends by the extension axiom to an automorphism of $V$ and is thus a morphism in $N_\F(V)$. Therefore, as $G$ is a model for $N_\F(V)$, we have $\widetilde{\alpha}=c_g|_{RT}$ for some $g\in C_G(T)$. We may choose $g$ to be a $p^\prime$-element. By Step~1, $[C_S(R),g]\leq [C_S(R),O^p(C_G(T))]\leq R$. Thus, $\alpha=c_g|_R$ extends to $\hat{\alpha}:=c_g|_{RC_S(R)}\in\Aut_\F(RC_S(R))$ with $[C_S(R),\hat{\alpha}]\leq R$. This proves the assertion.                                                                                                                                                                                                                                                                                                                                                                                                                                                                                                                                                                                                                                                                                                                                                                                                                                                                                                                                                                                                                                                                                                                               
\end{proof}

\begin{proof}[Proof of Theorem~\ref{MainCFE}]
 By Theorem~\ref{CFENormal}, $C_\F(\E)$ is normal in $\F$. Let $\D$ be a saturated subsystem of $\F$ over a subgroup $R\leq S$. 

\smallskip

If $\E$ and $\D$ centralized each other, then $\E\subseteq C_\F(R)$ and thus $R\leq C_S(\E)$. Moreover, $\D\subseteq C_\F(T)$. Since $\D$ is saturated, it follows from Alperin's Fusion Theorem and the Sylow axiom that $\D=\<O^p(\Aut_\D(P))\colon P\leq R\>_R$ is contained in $C_\F(\E)$.  

\smallskip

Assume now $\D\subseteq C_\F(\E)$. Then in particular $R\leq C_S(\E)$ and thus $\E\subseteq C_\F(R)$ by Theorem~\ref{MainCSE}(a). Moreover, $\D\subseteq C_\F(\E)\subseteq C_\F(T)$. Hence, $\E$ and $\D$ centralize each other.
\end{proof}

\begin{prop}\label{Coincide}
The subsystem $C_\F(\E)$ we defined coincides with the centralizer in $\F$ of $\E$ constructed by Aschbacher \cite[Chapter~6]{Aschbacher:2011}.
\end{prop}

\begin{proof}
The subgroup $C_S(\E)$ in our definition coincides with the one defined by Aschbacher, since it is in either case the largest subgroup $S_0$ of $S$ with $\E\subseteq C_\F(S_0)$. Write $\hat{C}_\F(\E)$ for the subsystem of $\F$ over $C_S(\E)$ which Aschbacher \cite[Chapter~6]{Aschbacher:2011} calls $C_\F(\E)$. Then by \cite[Theorem~4]{Aschbacher:2011}, $\Aut_{\hat{C}_\F(\E)}(P)=O^p(\Aut_{C_\F(T)}(P))\Aut_{C_S(\E)}(P)$ for every $P\in \hat{C}_\F(\E)^c\cap \hat{C}_\F(\E)^f$. Moreover, $\hat{C}_\F(\E)$ is normal in $\F$ and so in particular saturated. Hence, by Alperin's fusion theorem, $\hat{C}_\F(\E)\subseteq C_\F(T)$. Moreover, $\E\subseteq C_\F(C_S(\E))$, i.e. $\E$ and $\hat{C}_\F(\E)$ centralize each other. Thus, by Theorem~\ref{MainCFE}, we have $\hat{C}_\F(\E)\subseteq C_\F(\E)$. In particular, if $P\in C_\F(\E)^f\cap C_\F(\E)^c$, then $P\in \hat{C}_\F(\E)^f\cap \hat{C}_\F(\E)^c$. Arguing similarly as in the proof of Lemma~\ref{ShowWeaklyNormal}, it follows that $\Aut_{C_\F(\E)}(P)=\Aut_{C_S(\E)}(P)O^p(\Aut_{C_\F(T)}(P))$ for every $P\in C_\F(\E)^f$. Hence, for every $P\in C_\F(\E)^f\cap C_\F(\E)^c$, we have $\Aut_{C_\F(\E)}(P)=\Aut_{\hat{C}_\F(\E)}(P)$. Alperin's fusion theorem yields now $C_\F(\E)\subseteq \hat{C}_\F(\E)$ and thus $C_\F(\E)=\hat{C}_\F(\E)$.
\end{proof}

\section{Central products of normal subsystems}

\textbf{Throughout this section let $\F$ be a saturated fusion system over $S$, and let $\F_i$ be a normal subsystem of $\F$ over $S_i\leq S$ for $i=1,2$. Suppose furthermore $[S_1,S_2]=1$ and set $T=S_1S_2$.}

\smallskip

In this section we will prove Theorem~\ref{MainCentralProduct}. We will moreover show that $\F_1$ and $\F_2$ centralize each other if and only if $S_1\cap S_2\leq Z(\F_i)$ for $i=1,2$; this is particularly interesting when comparing Theorem~\ref{MainCentralProduct} to \cite[Theorem~3]{Aschbacher:2011}.

\smallskip
 
We will use throughout that $T$ is strongly $\F$-closed, since the product of two strongly $\F$-closed subgroups is always strongly $\F$-closed. This was first proved by Aschbacher \cite[Chapter 4]{Aschbacher:2011}; an alternative proof using factor systems was given by Craven \cite{Craven:2010}.

\smallskip

Crucial in the proof of Theorem~\ref{MainCentralProduct} is the following lemma.

\begin{lemma}\label{RadicalIntersect}
 Let $R\in \F^{cr}$. Then $R\cap T=(R\cap S_1)(R\cap S_2)$.
\end{lemma}

\begin{proof}
Clearly $(R\cap S_1)(R\cap S_2)\leq R\cap T$. Set $Q:=R\cap T$ and fix $i\in\{1,2\}$. Set $j=3-i$ and 
\[Q_i:=\{x_i\in S_i\colon \exists x_j\in S_j\mbox{ such that }x_ix_j\in Q\}.\]
It suffices to show that $Q_i\leq R$. Note that $[R,N_{Q_i}(R)]\leq R\cap S_i$. Moreover, for $x_i\in N_{Q_i}(R)$, by definition of $Q_i$, there exists $x_j\in S_j$ such that $y:=x_ix_j\in Q\leq R$. Then $c_{x_i}|_{R\cap S_i}=c_y|_{R\cap S_i}\in \Aut_R(R\cap S_i)$. So $\Aut_{Q_i}(R)$ is contained in 
\[X:=\{\phi\in\Aut_\F(R)\colon [R,\phi]\leq R\cap S_i\mbox{ and }\phi|_{R\cap S_i}\in\Aut_R(R\cap S_i)\}.\]
By \cite[Lemma~A.2]{Aschbacher/Kessar/Oliver:2011}, $X$ is a $p$-group. Moreover, $X$ is normal in $\Aut_\F(R)$, since $S_i$ is strongly closed. Hence, as $R$ is centric radical, it follows $\Aut_{Q_i}(R)\leq X\leq \Inn(R)$ and $N_{Q_i}(R)\leq R$. Observe that $R$ normalizes $Q_i$, since $R$ normalizes $Q$, $S_i$ and $S_j$. Hence, $RQ_i$ is $p$-group. As $R=N_{Q_i}(R)R=N_{Q_iR}(R)$, it follows thus $R=Q_iR$ and $Q_i\leq R$. This shows the assertion. 
\end{proof}

\begin{proof}[Proof of Theorem~\ref{MainCentralProduct}]
Suppose $\F_1$ and $\F_2$ centralize each other and set $\D:=\F_1*\F_2$. Recall that, by definition of $\F_1*\F_2$,
\[\D=\<\psi\in\Hom_\F(P_1P_2,T)\colon P_i\leq S_i\mbox{ and }\psi|_{P_i}\in\Hom_{\F_i}(P_i,S_i)\mbox{ for }i=1,2\>_T.\]
It follows from Proposition~\ref{P:F1F2Centralize} that $\D$ is the central product of $\F_1$ and $\F_2$, and thus in particular a saturated subsystem of $\F$. It remains to show that $\D$ is $\F$-invariant, and that the extension property for normal subsystems holds. 

\smallskip

\emph{Step~1:} We show that $\D$ is $\F$-invariant. As remarked above, $T$ is strongly closed. Let $\alpha\in\Aut_\F(T)$ and $P_i\leq S_i$ for $i=1,2$. As $S_i$ is strongly closed, we have $P_i^\alpha\leq S_i$ for each $i$. Moreover, if $\phi\in\Hom_\F(P_1P_2,T)$ with $\phi|_{P_i}\in\Hom_{\F_i}(P_i,S_i)$ for $i=1,2$, then $\phi^\alpha\in\Hom_\F(P_1^\alpha P_2^\alpha,T)$ and $(\phi^\alpha)|_{P_i^\alpha}=(\phi|_{P_i})^\alpha\in\Hom_{\F_i}(P_i^\alpha,S_i)$ as $\F_i$ is normal in $\F$. Thus, it follows from the construction of $\D=\F_1*\F_2$ that $\phi^\alpha$ is a morphism in $\D$. Moreover, we can conclude that  $\D^\alpha=\D$. Using the characterization of $\F$-invariant subsystems given in Proposition~\ref{Finvariant}(d), it remains now only to show that $\Aut_\D(R\cap T)\unlhd \Aut_\F(R\cap T)$ for every $R\in\F^{cr}$. Fix $R\in\F^{cr}$ and set $P:=R\cap T$. By Lemma~\ref{RadicalIntersect}, we have $P=P_1P_2$ where $P_i=R\cap S_i$ for $i=1,2$. By the construction of $\D$, we have therefore $\Aut_\D(P)=\{\phi\in\Aut_\F(P)\colon \phi|_{P_i}\in\Aut_{\F_i}(P_i)\mbox{ for }i=1,2\}$. As $\F_i$ is normal in $\F$ for $i=1,2$, it follows now easily that $\Aut_\D(P)$ is normal in $\Aut_\F(P)$ as required.

\smallskip

\emph{Step~2:} We show that the extension property holds for $\D$; i.e. fixing $\alpha\in\Aut_\D(T)$, we prove that $\alpha$ extends to $\hat{\alpha}\in\Aut_\F(TC_S(T))$ such that $[C_S(T),\hat{\alpha}]\leq T$. By definition of $\D$, we have $\alpha_i:=\alpha|_{S_i}\in\Aut_{\F_i}(S_i)$ for $i=1,2$. Now for each $i=1,2$, we can pick an extension $\hat{\alpha}_i\in\Aut_\F(S_iC_S(S_i))$ of $\alpha_i$ such that $[C_S(S_i),\hat{\alpha}_i]\leq S_i$ and $\hat{\alpha}_i|_{S_{3-i}}=\id_{S_{3-i}}$; this can be concluded from Lemma~\ref{CFCG0} or from Lemma~\ref{L:F1F2Centralize} and Lemma~\ref{ZCentralize}(a) below. Note that $TC_S(T)$ is weakly closed and contained in $S_iC_S(S_i)$ for $i=1,2$. Hence, $\hat{\beta}=(\hat{\alpha}_1|_{TC_S(T)})\circ (\hat{\alpha}_2|_{TC_S(T)})\in\Aut_\F(TC_S(T))$ is well-defined. Observe that $\hat{\beta}|_{S_i}=\hat{\alpha}_i|_{S_i}=\alpha_i=\alpha|_{S_i}$ for $i=1,2$. Hence, $\hat{\beta}|_T=\alpha$. Moreover, $[C_S(T),\hat{\beta}]\leq [C_S(T),\hat{\alpha}_1][C_S(T),\hat{\alpha}_2]\leq [C_S(S_1),\hat{\alpha}_1][C_S(S_2),\hat{\alpha}_2]\leq S_1S_2=T$. So Step~2 is complete. We conclude that $\D$ is normal in $\F$.
\end{proof}

If $\F_1$ and $\F_2$ centralize each other, then Theorem~\ref{MainCentralProduct} says basically that there is an explicitly constructed normal subsystem of $\F$ which is a central product of $\F_1$ and $\F_2$. Apart from the explicit construction, Aschbacher \cite[Theorem~3]{Aschbacher:2011} proves a similar result under the assumption that $S_1\cap S_2\leq Z(\F_i)$ for $i=1,2$. We will show that this assumption is actually equivalent to $\F_1$ and $\F_2$ centralizing each other. This is a  consequence of the next lemma; the reader might want to note that part (a) of this lemma was also cited in the proof of Theorem~\ref{MainCentralProduct} as an alternative to using Lemma~\ref{CFCG0}.

\begin{lemma}\label{ZCentralize}
Let $i\in\{1,2\}$ and set $j=3-i$. Suppose $S_1\cap S_2\leq Z(\F_i)$.
\begin{itemize}
 \item [(a)] Every automorphism $\beta\in\Aut_{\F_i}(S_i)$ extends to an automorphism $\hat{\beta}\in\Aut_\F(S_iC_S(S_i))$ with $[C_S(S_i),\hat{\beta}]\leq S_i$ and $\hat{\beta}|_{S_j}=\id_{S_j}$.
 \item [(b)] We have $\F_i\subseteq C_\F(S_j)$.
\end{itemize}
\end{lemma}

\begin{proof}
For the proof of (a) let $\beta\in\Aut_{\F_i}(S_i)$; we need to show that $\beta$ extends to an automorphism $\hat{\beta}$ as in (a). If $\beta=c_s|_{S_i}$ for some $s\in S_i$, then $\hat{\beta}=c_s|_{S_iC_S(S_i)}$ is an extension of $\beta$ with $[C_S(S_i),\hat{\beta}]\leq [S,S_i]\leq S_i$ and $\hat{\beta}|_{S_j}=c_s|_{S_j}=\id_{S_j}$. So if $\beta\in\Inn(S_i)$, then there exists an extension $\hat{\beta}$ with the required properties. By the Sylow axiom, $\Inn(S_i)$ is a Sylow $p$-subgroup of $\Aut_{\F_i}(S_i)$. Hence, we may assume that $\beta$ is a $p^\prime$-automorphism. By the extension property for normal subsystems, $\beta$ extends to $\hat{\beta}\in\Aut_\F(S_iC_S(S_i))$ with $[C_S(S_i),\hat{\beta}]\leq S_i$. As $\beta$ is a $p^\prime$-automorphism, replacing $\hat{\beta}$ be a suitable power of itself, we may assume that $\hat{\beta}$ is a $p^\prime$-automorphism as well. Note that $S_j\leq C_S(S_i)$. As $S_j$ is strongly closed, we conclude $[S_j,\hat{\beta}]\leq S_i\cap S_j=S_1\cap S_2$. Since $S_1\cap S_2\leq Z(\F_i)$, we have $\hat{\beta}|_{S_1\cap S_2}=\beta|_{S_1\cap S_2}=\id_{S_1\cap S_2}$. Hence, as $\hat{\beta}$ is a $p^\prime$-automorphism, we have $[S_j,\hat{\beta}]=[S_j,\hat{\beta},\hat{\beta}]\leq [S_1\cap S_2,\hat{\beta}]=1$. This shows (a). In particular, $\Aut_{\F_i}(S_i)\subseteq C_\F(S_j)$. As $S_j$ is strongly $\F$-closed and thus weakly $\F$-closed, Proposition~\ref{WeaklyClosedCentralized} yields $\F_i\subseteq C_\F(S_j)$. So (b) holds.
\end{proof}

\begin{prop}\label{NormalCentralizeEachOther}
The normal subsystems $\F_1$ and $\F_2$ centralize each other if and only if $S_1\cap S_2\leq Z(\F_i)$ for $i=1,2$. 
\end{prop}

\begin{proof}
This follows from Lemma~\ref{L:F1F2Centralize} and Lemma~\ref{ZCentralize}(b).
\end{proof}

\bibliographystyle{amsplain}
\bibliography{repcoh}

\end{document}